\documentclass[review]{elsarticle}

\journal{Journal of Geometry and Physics}

\usepackage[utf8]{inputenc}

\usepackage[normalem]{ulem}
\useunder{\uline}{\ul}{}
\usepackage{lineno,hyperref}
\usepackage{amsmath}
\usepackage{amssymb}
\usepackage{MnSymbol}
\usepackage{mathtools}
\usepackage{amsfonts}
\usepackage{color}

\definecolor{amber}{rgb}{1.0, 0.49, 0.0}
\newcommand{\stefano}[1]{{ \color{black}  #1}}

\newcommand{\fps}[1]{{ \color{black}  #1}}

\usepackage{amsthm}
\usepackage{subcaption}

\newtheorem{remark}{Remark}

\newtheorem{theorem}{Theorem}

\title{Energetic decomposition of Distributed Systems with Moving Material Domains: the port-Hamiltonian model of Fluid-Structure Interaction}

\author{Federico Califano\textsuperscript{1,*}\corref{RAM}, Ramy Rashad\textsuperscript{1}, Frederic P. Schuller\textsuperscript{2}, Stefano Stramigioli\textsuperscript{1}
}
\address{\textsuperscript{1} Robotics and Mechatronics Department, University of Twente, The Netherlands}
\address{\textsuperscript{2} Department of Applied Mathematics, University of Twente, The Netherlands}
\cortext[RAM]{Corresponding author. Email: f.califano@utwente.nl}

\begin{document}

\begin{abstract}
We introduce the geometric structure underlying the port-Hamiltonian models for distributed parameter systems exhibiting moving material domains. The first part of the paper aims at introducing the differential geometric tools needed to represent infinite-dimensional systems on time--varying spatial domains in a port--based framework. A throughout description on the way we extend the structure presented in the seminal work \cite{VanDerSchaft2002}, where only fixed spatial domains were considered, is carried through. As application of the proposed structure, we show how to model in a completely coordinate-free way the 3D fluid--structure interaction model for a rigid body immersed in an incompressible viscous flow as an interconnection of open dynamical subsystems.
\end{abstract}

\maketitle

\newcommand{\pair}[2]{\left\langle \left.  #1  \right|  #2 \right\rangle}

\newcommand{\weakpair}[2]{\left\langle  #1 , #2 \right\rangle}
\newcommand{\doubleweakpair}[2]{\left\langle \left\langle  #1 , #2 \right\rangle \right\rangle}

\newcommand{\gothg}{\mathfrak{g}}
\newcommand{\gothgV}{\mathfrak{g^v}}
\newcommand{\gothgstar}{\mathfrak{g}^{*}}
\newcommand{\extcovd}{\extd_{\nabla}}
\newcommand{\dotwedge}{\dot{\wedge}}

\newcommand{\extd}{\textrm{d}}

\newcommand{\goths}{\mathfrak{s}}
\newcommand{\gothsStar}{\mathfrak{s}^{*}}

\newcommand{\secTanBdl}[1]{\Gamma(T#1)}

\newcommand{\volF}{\mu_\text{vol}}

\newcommand{\abar}{{\bar{a}}}


\newcommand{\LieD}[2]{\cl{L}_{#1}#2}
\newcommand{\dt}{\frac{d}{dt}}

\newcommand{\dtLine}[1]{\frac{d}{dt}\bigg|_{#1}}
\newcommand{\depsLine}[1]{\frac{d}{d \epsilon}\bigg|_{#1}}

\newcommand{\pt}{\frac{\partial}{\partial t}}
\newcommand{\JLBrack}[3]{\llbracket{ #2},{#3} \rrbracket_{#1}}

\newcommand{\inner}[2]{\left \langle #1 , #2 \right\rangle}

\newcommand{\Lbrack}[2]{\left[#1,#2 \right]}
\newcommand{\LbrG}[2]{\Lbrack{#1}{#2}_{\gothg}}
\newcommand{\LbrS}[2]{\Lbrack{#1}{#2}_{\goths}}

\newcommand{\Pbrack}[2]{\left \{#1,#2 \right\}}

\newcommand{\parXi}[1][i]{\frac{\partial}{\partial x^{#1}}}
\newcommand{\dXi}[1][i]{\extd{ x^{#1}}}

\newcommand{\varD}[2]{\frac{\delta {#1}}{\delta {#2}}}

\newcommand{\adS}[1]{\boldsymbol{ad}_{#1}}
\newcommand{\adSdual}[1]{\boldsymbol{ad}^*_{#1}}

\newcommand{\spKForm}[2]{\Omega^{#1}(#2)}
\newcommand{\spVecF}[1]{\Gamma(T#1)}
\newcommand{\spVecX}[2]{\mathfrak{X}_{#2}(#1)}
\newcommand{\spKFormM}[1]{\spKForm{#1}{M}}
\newcommand{\spKFormMbc}[2]{\Omega^{#1}_{#2}(M)}
\newcommand{\spVecM}{\spVecX{M}{}}
\newcommand{\spVecMtrace}{\spVecX{i(\partial M)}{}}
\newcommand{\spVecMpartial}{\spVecX{\partial M}{}}
\newcommand{\spVecMt}{\spVecX{M}{t}}
\newcommand{\spFn}[1]{C^\infty(#1)}

\newcommand{\vMeas}[1]{\textbf{v}(#1)}
\newcommand{\mMeas}[1]{\textbf{m}(#1)}
\newcommand{\mForm}{\mu_t}
\newcommand{\divr}[1]{\text{\normalfont div}(#1)}
\newcommand{\sdiff}{D_\mu(M)}
\newcommand{\RVF}[1]{\cl{R}(#1)}
\newcommand{\connect}[1]{\stackrel{\bb{#1}}{\nabla}}
\newcommand{\adIflat}[2]{ad_{#1}^*(\bb{I}^\flat {#2})}
\newcommand{\vf}{{\tilde{v}}}
\newcommand{\uf}{\tilde{u}}
\newcommand{\omV}{{\hat{\omega}}}

\newcommand{\mNmOne}{(-1)^{n-1} }

\newcommand{\effOne}{\delta_{\mu}H}
\newcommand{\effTwo}{\delta_{\alpha}H}
\newcommand{\bound}[1]{i^*({#1})}
\newcommand{\effOneB}{\bound{\delta_{\mu}H}}
\newcommand{\effTwoB}{\bound{\delta_{\alpha}H}}

\newcommand{\gtInv}{g^{-1}_t}
\newcommand{\diffG}{\mathcal{D}(M)}

\newcommand{\parL}[1]{\frac{\delta l}{\delta {#1}}}

\newcommand{\Mbound}{\partial M}

\newcommand{\Ltwo}{\cl{L}_2}
\newcommand{\Mflat}{\bb{M}^\flat}
\newcommand{\Msharp}{\bb{M}^\sharp}

\newcommand{\state}{(\alpha,\mu)}

\newcommand{\B}[1]{\boldsymbol{#1}}
\newcommand{\bb}[1]{\mathbb{#1}}
\newcommand{\cl}[1]{\mathcal{#1}}
\newcommand{\Rn}{\bb{R}^n}

\newcommand{\tens}{%
  \mathbin{\mathop{\otimes}}%
}

\newcommand{\todo}[1]{{\color{red}{#1}}} 


\newcommand{\half}{\frac{1}{2}}

\newcommand{\TwoTwoMat}[4]{
\begin{pmatrix}
#1 & #2 \\
#3 & #4
\end{pmatrix}}
\newcommand{\TwoVec}[2]{
\begin{pmatrix}
#1\\
#2 
\end{pmatrix}}

\newcommand{\ThrVec}[3]{
\begin{pmatrix}
#1\\
#2\\
#3
\end{pmatrix}}

\newcommand{\map}[3]{{#1}:{#2}\rightarrow {#3}}
\newcommand{\fullmap}[5]{
\begin{split}
#1 : {#2} &\rightarrow {#3}\\
	{#4} &\mapsto {#5}
\end{split}
}

\newcommand{\fullmapNoName}[4]{
\begin{split}
 	{#1} &\rightarrow {#2}\\
	{#3} &\mapsto {#4}
\end{split}
}

\newcommand{\trace}{\text{\normalfont tr}}

\newcommand{\blue}{\color{blue}}

\newcommand{\pH}{port-Hamiltonian }

\newcommand{\DkT}{\cl{\tilde{D}}_k}

\newcommand{\fstr}{\text{f}_\text{s}}
\newcommand{\eBoud}{e_{\partial k}}
\newcommand{\fBoud}{f_{\partial k}}

\newcommand{\eSV}{{e}_{sk}}
\newcommand{\fSV}{{f}_{sk}}

\newcommand{\ramy}[1]{{\color{black}#1}}
\newcommand{\DiracEuler}{\mathcal{D}_\text{E}}
\newcommand{\PsiB}{\Psi_\text{b}}
\newcommand{\PsiV}{\Psi_\text{v}}

\section{Introduction}
The goal of this work is to extend the existing port-Hamiltonian theory of open distributed parameter systems on fixed spatial domains to the practically relevant case where the spatial domains can vary in time. Building on the foundational exposition  \cite{VanDerSchaft2002} of the theory for fixed domains, also the extended theory presented here will be cast in proper differential geometric language. One intended application pushing this extension of the theory is a complete port-Hamiltonian study of fluid-solid interactions, which we are ultimately able to provide in the second part of the paper. 

Port-Hamiltonian theory in general, quite apart from the intricacies posed by distributed parameter systems or moving domains for the latter,
generalizes the Hamiltonian description for inherently closed dynamical systems to the case of open systems \cite{Geoplex}. On the one hand, the theory allows to describe an open system entirely in its own right, without any assumptions about the environment. It does so employing the concept of ports, which give the theory its name. A port is a pair of suitably chosen dual quantities that appear in the description of an open system. It is via these ports that an open system can gain or lose energy, which manifests itself in the fact that the canonical duality product between the dual port variables yields the rate of change, or power, of the energy transfer. On the other hand, port-Hamiltonian theory also provides the mathematical formalism of how an open system is coupled to other open systems through their respective ports. The central construction here is a so-called Dirac structure, which abstractly speaking sits in between the open systems and routes all energy flows, which come in and flow out through the various ports, such that the total energy is conserved. A system composed in such a way of several open subsystems and an interconnecting Dirac structure may itself be still open, namely if some ports are left uncoupled, or closed.  

When port-Hamiltonian theory is applied to systems whose parameters are distributed over a spatial domain (we refer to \cite{rashad2020twenty} for a literature review on infinite-dimensional port-Hamiltonian systems), any possible port falls into one of two classes: {\it distributed ports} which channel power flows within the spatial domain of an open system and {\it boundary ports} which channel power flows through the boundary of the same spatial domain. Thus it is no surprise that the consideration of moving domain boundaries results in novel boundary port variables compared to the previously studied fixed boundary case. The sophistication required of the underlying mathematical theory varies between different types of distributed parameter systems. As we delineate in Section \ref{sec:stresstensor}, the highest demands come from open systems in continuum mechanics, both solid or fluid, and require the use of nested manifolds and generalized tensor-valued forms thereon in order to describe the non-trivial energy transfer across domain boundaries. 

The idea of a moving domain boundary,
and a version of the associated port variables, can also be studied in a technically simpler setting than the one required by the intricacies of continuum mechanics, providing a way to extend port-Hamiltonian theory to the study of moving domain boundaries for some systems of relevance introduced in \cite{VanDerSchaft2002}, such as the electromagnetic field and ideal fluid dynamics.
We start our theoretical constructions in Section \ref{sec:newPairing} by employing such a simpler setting, in order to disentangle the concept of a moving domain boundary from the technical refinements provided later. 

As a signature application of our techniques to treat moving domain boundaries, we study the boundary interconnection of a rigid body with a viscous fluid and provide a complete energetic decomposition of this system in port-Hamiltonian fashion. In particular, we show how the novel port-based representation of the motion of the spatial domains is necessary to correctly implement a {\it no-slip} condition between the fluid and the solid.
One significant advantage of this  decomposition of the system --- into separate open subsystems and the precise structure of their interconnection via a Dirac structure --- is its modularity: When needed, the overall port-Hamiltonian model can be updated by replacement, addition or removal of subsystems in order to provide either more specialised or arbitrarily more sophisticated models than the one considered here. Another advantage is that the complete port-Hamiltonian decomposition allows to analyse the power flow between all system components in order to derive conclusions on the stability of a sophisticated multi-component system \cite{Geoplex} or to devise novel numerical algorithms that exploit the associated conservation laws subsystem by subsystem \cite{califano2021decoding}. In other words, the present work allows to extend the port-Hamiltonian description of our previous fluid models \cite{rashad2020porthamiltonian1,rashad2020porthamiltonian2,califano2021geometric} to the fluid-structure interaction model in Section \ref{sec:pH_fsi}. 

It behoves us to point out pertinent previous work in the direction of our contributions. Partial results on moving boundaries within the port-Hamiltonian framework have been obtained in \cite{Diagne2013}, which provides a non-geometric treatment of one-dimensional domains with moving boundaries by consideration of a moving interface that separates two subsystems and is subject to its own dynamics, and in \cite{Vu2016}, which develops a port-Hamiltonian model that includes the moving material domains of a Tokamak and is conceptually and technically situated at roughly the level of our partial analysis in section \ref{sec:newPairing}. Various simplified port-Hamiltonian models of fluid-structure interactions, tailored for particular applications and in non-geometric formulation, can be found in \cite{mora18,mora20,cardoso15}. Standard Hamiltonian and Lagrangian treatments of systems within the scope of the present paper, such as \cite{kanso2005locomotion,vankerschaver2010dynamics,jacobs2012fluid,glass2012movement,planas2014viscous,fusca2018groupoid}, touch upon various technical issues also of relevance for our port-Hamiltonian treatment, but are inherently constrained to consider closed systems and thus cannot resolve the energetic subsystems and the energy routing between them. 

\paragraph{Notation} 
Throughout the paper, we use \textit{bond graphs} \cite{Geoplex} for the graphical representation of systems in the port-Hamiltonian formalism, in order to aid those readers familiar with this sophisticated diagrammatical language, but the paper can be read and understood without them.
\fps{The other mathematical symbols used in this paper are as follows.  The mathematical model of fluid domains is a compact, orientable, $n$-dimensional Riemannian manifold $M$ with (possibly empty) boundary $\partial M$. The} space of vector fields on $M$ is the space of sections of the tangent bundle $TM$, that will be denoted by $\spVecM$. The space of differential $p$-forms is denoted by $\Omega^p(M)$ \fps{and we also occasionally refer to $0$-forms as functions, $1$-forms as covector fields} and $n$-forms as top forms. For any
$v\in \spVecM$, we use the standard definitions for the interior product by $\iota_v:\Omega^p(M) \to \Omega^{p-1}(M)$ and the Lie derivative operator $\mathcal{L}_v$ acting on tensor fields of any valence. 
The Hodge star operator $\star:\Omega^p(M) \to \Omega^{n-p}(M)$ and the associated volume form $\volF=\star 1$ as well as the musical operators $\flat:\spVecM \to \Omega^{1}(M)$ and $\sharp:\Omega^{1}(M) \to \spVecM$, which respectively transform vector fields to 1-forms and vice versa, are all uniquely induced by the Riemannian metric in the standard way. 
When making use of Stokes theorem $\int_M \extd \omega=\int_{\partial M} \textrm{tr}(\omega)$ for $\omega \in \Omega^{n-1}(M)$. The trace operator $\textrm{tr}:=i^*$ is the pullback of the\fps{canonical} inclusion map $i: \partial M \hookrightarrow M$.
When dealing with tensor-valued forms we adopt the additional convention that a numerical index $i\in\{1,2\}$ on the left or the right of a standard operator on differential forms indicates whether the operator acts on the ``first leg'' (the tensor value) or on the ``second leg'' (the underlying form) of the tensor-valued form on the respective side of the operator: For an $n$ form-valued $m$ form $\alpha\otimes\beta$, for instance, we define $\star_1(\alpha \tens \beta):=\star\alpha \tens \beta$ and $\star_2(\alpha \tens \beta):=\alpha \tens \star\beta$. 
For the representation of the above concepts in terms of coordinate charts, see for instance \cite{gilbert2019geometric}. We will introduce along the paper more advanced operators and constructions where needed.

\section{Energy Continuity and Geometric Reynolds Transport theorem}
\label{sec:newPairing}
In this section, after briefly resuming the \pH structure underlying the classical model on fixed domains present in \cite{VanDerSchaft2002}, we introduce a simple way to represent the power port corresponding to moving domains. The mathematical technology to express this port will need to be refined when tackling continuum mechanics, which will be done in the next section, but is still powerful enough to represent the moving domain version of the systems studied in \cite{VanDerSchaft2002}, such as the electromagnetic field and ideal fluid dynamics. 

\subsection{Fixed Domains}
In \cite{VanDerSchaft2002} the port-Hamiltonian formulation of distributed parameter systems is given on a fixed $n$-dimensional Riemannian manifold $M$. A distributed physical system is characterised by an energy density $\mathcal{H}:\cl{X}\rightarrow \Omega^n(M)$, an extensive variable that produces the total energy (the \textit{Hamiltonian} functional) of the system once integrated over its spatial domain, i.e. $H[x]:=\int_M \mathcal{H}(x)$. We assume the density $\mathcal{H}$ to not depend explicitly on time, but only on the so-called \textit{energy variables} $x:=(x_1,\cdots,x_m)\in \cl{X}$, where $x_i\in \Omega^{k_i}(M)$ are differential forms of appropriate degrees, for $i\in \{1,\cdots,m\}$. 
In this setting the variation of total energy is taken as
\begin{equation}
\label{eq:DensityVariationFixedBoundary}
   \dot{H}=\int_M \dot{\mathcal{H}} = \int_M \delta_{x_1}H \wedge \dot{x}_1 + \cdots + \delta_{x_m}H \wedge \dot{x}_m ,
\end{equation} 
where $\delta_{x_i}H \in \Omega^{n-k_i}(M)$ denotes the variational derivative of the Hamiltonian with respect to the energy variable $x_i$, that can be shown to be a differential form of complementary degree with respect to the associated energy variable $x_i$. In this sense, the pH formulation in \cite{VanDerSchaft2002} is based on the duality product expressed by (\ref{eq:DensityVariationFixedBoundary}) by means of the wedge product. In \cite{VanDerSchaft2002} the whole construction is limited to $m=2$ and energy variables being a $p$-form and $q$-form, such that their dynamics could be defined in terms of a canonical pH model based on exterior derivative operators\footnote{This choice corresponds to a system of two conservation laws, with the constraint $p+q=n+1$.}. This model would encode, together with the dynamics of the single energy variables, the power continuity equation, that is the first principle of thermodynamics.
In fact the pH system encodes the power equality
\begin{equation}
\label{eq:continuityFixedBoundary}
\dot{H}=\int_M \dot{\mathcal{H}}=\int_M (\extd \Phi + \sigma)=\int_{\partial M} \textrm{tr} (\Phi) +\int_M \sigma
\end{equation}
for some $\Phi \in \Omega^{n-1}(M)$ (\textit{in/out power flux}) and $\sigma \in \Omega^n(M)$ (\textit{power source/sink}), representing energy continuity equation at an integral level.
As standard in the pH approach, $\textrm{tr}( \Phi)$ and $\sigma$ are expressed as duality products of novel defined dual port variables (called \textit{effort} and \textit{flow}), of respectively boundary and distributed type. In particular the port Hamiltonian formulation allows to represent the boundary power flux $\textrm{tr}(\Phi) \in \Omega^{n-1}(\partial M)$ as the duality product $e_\Phi \wedge f_\Phi$ for some differential forms on $\partial M$ whose degrees summed together is $n-1$. Similarly the source term $\sigma$ is expressed as the duality product $e_\sigma \wedge f_\sigma$ in which the variables are differential forms defined on $M$ and the sum of their degrees is $n$. The power port corresponding to the storage element is characterised by effort $e_s=\delta_x H$ and flow $f_s=-\dot{x}$, where $x$ collects all the energy variables.
The pH model is then characterised by the following power continuous property
\begin{equation}\label{eq:Dirac_power_balance_1}
\underbrace{\int_{M} e_s \wedge f_s}_{-\dot{H}} + \underbrace{\int_{\partial M} e_\Phi \wedge f_\Phi}_{\textrm{Boundary Power Flow}} + \underbrace{\int_{M} e_\sigma \wedge f_\sigma}_{\textrm{In-domain Power Flow}}=0,
\end{equation}
resembling  (\ref{eq:continuityFixedBoundary}).

\subsection{Moving Domains} 
A limitation in \cite{VanDerSchaft2002} is the intrinsic view of the distributed parameter system, i.e., even if proper boundary terms $e_{\Phi}$ and $f_{\Phi}$ pop out at a power balance level, the spatial domain is constraint to be fixed. For modelling physical systems whose spatial interconnection produces a relative motion of their common interface (which is what happens e.g., in a FSI system) it is necessary to adopt an \textit{extrinsic} geometric modelling approach. This means that we will consider the case in which the spatial domain $M$ is embedded in a bigger, fixed ambient space, in which it is possible to describe a motion $M(t)$ for every time instant $t$ (see Sec. \ref{sec:fsi}). In order to extend the model to account for moving spatial domains one needs to generalise the energy continuity equation (\ref{eq:DensityVariationFixedBoundary}) to the version that considers a moving domain $M(t)=: M_t$ (with moving boundary $\partial M(t)=:\partial M_t$), that reads \cite{frankel2011geometry}
\begin{equation}
\label{eq:DensityVariationMovingBoundary}
   \dot{H}=\int_{M_t} \dot{\mathcal{H}} + \mathcal{L}_u \mathcal{H}, 
\end{equation} 
where $u \in \spVecX{M_t}{}$ is the vector field representing the spatial motion of $M_t$. This expression, which is an application of \textit{Reynolds transport theorem} to the energy density of the system, establishes a new boundary port that allows for power flows in the system. In fact, using the fact that $\mathcal{H}$ is a top-form, applying \textit{Cartan's magic formula} and Stokes theorem, the power continuity equation (\ref{eq:DensityVariationMovingBoundary}) can be rewritten as
\begin{equation}
\label{eq:DensityVariationMovingBoundary2}
   \dot{H}=\int_{M_t} \dot{\mathcal{H}} + \extd \iota_u \cl{H} =\int_{M_t} \dot{\mathcal{H}}+ \int_{\partial M_t}\textrm{tr}(\iota_u  \mathcal{H}).
\end{equation} 
Using the identity 
\begin{equation}
\label{eq:hirani}
    \iota_{u}  \mathcal{H}=\star \mathcal{H} \wedge \star u^{\flat},
\end{equation}
where $u^{\flat}:=\flat(u)$, and distributing the trace over the wedge product, we can model the mechanism as a boundary port describing the evolution in time of the boundary $\partial M_t$, in which the boundary port variables $(e_u,f_u)$ and their pairing are defined as
\begin{align}
\label{eq:boundaruEffort1}
    e_u:=\textrm{tr} (\star \mathcal{H})\in \Omega^0(\partial M),\\
    \label{eq:boundaruFlow1}
    f_u:=\textrm{tr}(\star u^{\flat})\in \Omega^{n-1}(\partial M),\\
    \label{eq:boundaryPairing1}
    \langle e_u|f_u \rangle:=\int_{\partial M_t}e_u \wedge f_u.
\end{align}

\begin{figure}
\centering
\includegraphics[width = 0.9\textwidth]{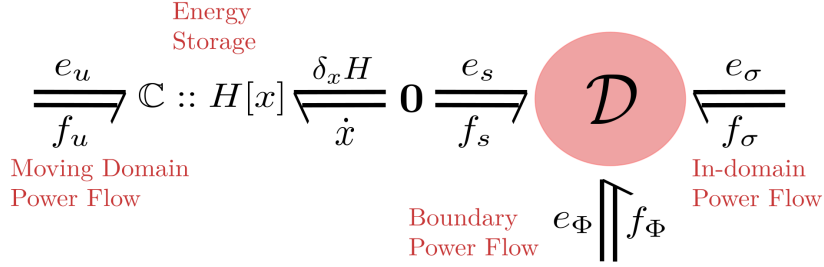}
\caption{Graphical representation of a distributed port-Hamiltonian system on a moving spatial domain with three open ports characterising exchanged external power between the system and the external world. Notice that the bondgraph for the system with fixed domain is identical, but without the moving domain port.}
\label{fig_pH_model_moving_domain}
\end{figure}

From a generalised Bond-Graph perspective, as shown in Fig. \ref{fig_pH_model_moving_domain}, the $\bb{C}$-element corresponding to the energy functional $H$, becomes a \textit{dual port storage element}, on which the flow variable $f_u$ contains information on how the boundary is moving while the effort $e_u$ is just the energy function evaluated at the boundary. Their pairing represents then the variation of the total energy due to the movement of the boundary. Therefore,  we can rewrite (\ref{eq:DensityVariationMovingBoundary}) as
\begin{equation}\label{eq:Hdot_C_element}
	 \dot{H}= \int_{M_t} \delta_x H \wedge \dot{x} + \int_{\partial M_t} e_u \wedge f_u,
\end{equation}
which defines the energy stored represented graphically by the $\bb{C}$-element in Fig. \ref{fig_pH_model_moving_domain}.

In this general setting, and using the pH model corresponding to the specific physical system to calculate the term $\dot{\mathcal{H}}$ as in (\ref{eq:continuityFixedBoundary}), the energy continuity equation at an integral level
generalises to
\begin{equation}
\label{eq:continuityMovingBoundary3}
\dot{H}=\int_{M_t} \dot{\mathcal{H}}+\mathcal{L}_u \mathcal{H}=\int_{M_t} (\extd (\Phi+\iota_u\mathcal{H}) + \sigma)=\int_{\partial M_t} \textrm{tr} (\Phi+\iota_u \mathcal{H}) +\int_{M_t} \sigma,
\end{equation}
where, as for the computation of the second term in (\ref{eq:DensityVariationMovingBoundary2}), we used Cartan's magic formula on the Lie derivative. Similarly to \cite{Vu2016}, we recognise the quantity $\Phi+\iota_u \mathcal{H}=:\Phi_{\textrm{rel}}$ 
as the \textit{relative power flux} to the observer, i.e., to the moving spatial domain $M_t$. 

This means that the total energy within the domain will clearly depend on the state variable fields within the domain, considering that the energy density is not directly dependent on time, and on the motion of the spatial domain itself.
The given definition of the boundary port describing the power injection due to a moving spatial domain has the advantage of being completely characterised by scalar-valued (or standard) differential forms and the exterior wedge product. It follows that the power continuous property encoded by the model described in \cite{VanDerSchaft2002} is extended without the use of new mathematical tools. The new port-Hamiltonian model is then characterised by the following power continuous property, which combines the power flows due to internal dynamics and the effect of moving spatial domain:
\begin{equation}\label{eq:moving_domain_power_balance}
\dot{H} = \underbrace{\int_{\partial M_t} e_\Phi \wedge f_\Phi}_{\textrm{Boundary Power Flow}}  + \underbrace{\int_{\partial M_t} e_u \wedge f_u}_{\textrm{Moving Domain Power Flow}}+ \underbrace{\int_{M_t} e_\sigma \wedge f_\sigma}_{\textrm{In-domain Power Flow}},
\end{equation}
resembling indeed (\ref{eq:continuityMovingBoundary3}). Using the definitions of $e_s=\delta_x H$ and $f_s=-\dot{x}$ and (\ref{eq:Hdot_C_element}), we can rewrite (\ref{eq:moving_domain_power_balance}) as
\begin{equation}\label{eq:Dirac_power_balance_2}
	\int_{M_t} e_s \wedge f_s + \int_{\partial M_t} e_\Phi \wedge f_\Phi + \int_{M_t} e_\sigma \wedge f_\sigma=0,
\end{equation}
which is identical to (\ref{eq:Dirac_power_balance_1}).
This power balance is encoded graphicaly in Fig. \ref{fig_pH_model_moving_domain} by the so called Stokes-Dirac structure $\cl{D}$ whose exact definition depends on the physical system described by the distributed \pH system as will be demonstrated later.

\section{Stress tensor in continuous mechanics: a new pairing is needed}
\label{sec:stresstensor}
In the previous section we reviewed the \pH structure based on the $\wedge$ pairing introduced in \cite{VanDerSchaft2002} and extended it to the case of moving domains without introducing new mathematical tools, i.e., the newly introduced power port is correctly modelled by the $\wedge$ pairing introduced in \cite{VanDerSchaft2002}. In the following we show that:
\begin{enumerate}
    \item A new pairing is needed to correctly capture the boundary power flow in continuous mechanics, regardless of the underlying spatial manifold being fixed or moving;
    \item This new pairing will induce a new way to represent the previously introduced moving domain power port, which will be crucial to implement the \textit{no-slip} condition in a FSI system.
    \item Specialising the proposed construction to the case of viscous fluid dynamics will allow to represent a geometric \pH model of FSI in the next section.
\end{enumerate}

The reason why a new pairing is needed to represent boundary power flow in continuous mechanics, and as a consequence the bond space in the pH framework needs to be extended, is due to the tensorial nature of \textit{stress}. We remark that this extension is technically present in \cite{califano2021geometric} where the bond space in the pH framework for Newtonian fluids is presented on a fixed domain, but we shall give a clear motivation in the sequel, abstracting from a specific physical system. The reason why this extension is necessary is that stress does not possess a representation as standard differential form, but needs to be a covector-valued $(n-1)$-form. The motivation behind this necessity is well addressed e.g., in \cite{frankel2011geometry,kanso2007geometric} and the underneath intuition is that stress must be geometrically an object to be integrated on a surface (the $(n-1)$-form) to get a covector, i.e. the traction force. In the following we give rigorous definitions and constructions building on this idea.
\subsection{Stress as covector-valued form and its trace}

Let as usual $M$ be a smooth manifold. We indicate with $\chi^{(r,s)}(M)$ the $C^\infty(M)$-module of tensor fields of valence $(r,s)$ and $\Omega^{p}(M)$ the submodule of $p$-forms. As examples we remind $\chi^{(1,0)}(M)=\spVecX{M}{}$ and $\chi^{(0,1)}(M)=\Omega^1(M)$. A \textit{tensor-valued form} of valence $(r,s,p)\in\mathbb{N}^3$ on $M$ is then an element of the tensor product
$$\chi^{(r,s)}(M) \otimes \Omega^{p}(M)$$
of modules. 

In continuum mechanics, the stress tensor is a covector-valued $(n - 1)$-form $\cl{T}\in \Omega^1(M) \tens \Omega^{n-1}(M)$, or equivalently a tensor-valued form on $M$ of valence $(0,1,n-1)$.
For formal correctness we remark that integration of $\cl{T}$ over a $(n-1)$-dimensional surface is geometrically meaningless in a curved space where tensors attached to different points cannot be meaningfully summed and therefore also not integrated \cite{Marsden2012}. However it is interesting to notice that when computing the \textit{stress power}, defined by the pairing of $\cl{T}$ with the velocity field of the continuum, only integration of the scalar power density is considered. In fact it is possible to define the (metric independent) stress power on $\partial M$:
\begin{equation}
\label{eq:SurfacePowerStress}
    \text{Stress Power}=\int_{\partial M} \textrm{tr}(\iota_v{\cl{T}}),
\end{equation}
 where $v\in \spVecM$ is the vector field corresponding to the macroscopic velocity field of the continuum. 
The key observation to understand the necessity of extending the pH structure in this context is that identity (\ref{eq:hirani}) cannot be applied, since $\cl{T}$ is not a top-form. As a consequence it is not possible to express the pairing $\iota_v \cl{T}$ in a form involving the $\wedge$ operator, and thus it is not possible to distribute the trace operator over the stress power density to define boundary effort and flow variables characterised by the usual bond-space based on the $\wedge$ product.

For an efficient construction of  boundary ports in continuum mechanics we need a refinement of the notion of tensor-valued form. Rather than being defined on one ordinary smooth manifold $M$, this refinement takes place on a \textit{nested manifold} $(U,M,f)$ which we define by specification of two smooth manifolds and an injective smooth map $f: U \to M$ between them. Note that any smooth manifold $M$ gives rise to a trivial nested manifold $(M,M,\textrm{id}_M)$ and that any smooth manifold $M\supseteq \partial M$ with boundary gives rise to a nested manifold $(\partial M, M, i)$ where $i: \partial M \to M$ is the canonical inclusion map. Indeed, these two examples of nested manifolds are precisely the two cases that play a role in the constructions of this paper.  

A \textit{generalized tensor-valued form} of valence $(r,s,p)\in\mathbb{N}^3$ on a nested manifold $(U,M,f)$ is an element of the set
$$\chi^{(r,s)}(M) \otimes_f \Omega^{p}(U)\,,$$
which denotes that subset of $\chi^{(r,s)}(M) \otimes \Omega^{p}(U)$ whose elements $\sigma$ satisfy the condition
$$\sigma(u) \in (T^*_{f(u)}M)^{\otimes r} \otimes (T_{f(u)}M)^{\otimes s} \otimes (T^*_uU)^{\wedge p} \}\qquad\textrm{for all } u \in U\,,$$
where a superscripts ${\otimes r}$ or $\wedge p$ to the right of a vector space indicate, respectively, the $r$-fold tensor product or $p$-fold antisymmetric tensor product of that vector space.
Pointwise definition of the relevant operations makes the set $\chi^{(r,s)}(M) \otimes_f \Omega^{p}(U)$ into a $C^\infty(U)$-module, and as such we will use it. 
Note that a generalized tensor-valued form $\sigma$ on the trivial nested manifold $(M,M,\textrm{id}_M)$ is just a tensor-valued form. This is not the case, however, for a nested manifold of the form $(\partial M, M, i)$; the generalized tensor-valued forms on there cannot be expressed in terms of ordinary tensor-valued forms on $M$. The partial trace of tensor-valued forms, which is defined in the following, provides an example for a generalized covector-valued form on the non-trivial nested manifold $(\partial M, M, i)$. 

The \textit{partial trace} operator
$$
\textrm{ptr}: \chi^{(r,s)}(M)\otimes\Omega^p(M) \longrightarrow \chi^{(r,s)}(M)\otimes_i \Omega^p(\partial M)\,,\qquad
\sigma \mapsto \textrm{ptr}(\sigma)
$$
maps an ordinary tensor-valued form on a smooth manifold $M$ to a generalized tensor-valued form of the same degree on the nested manifold $(\partial M, M, i)$ induced by $M$ and is defined such that for all $\alpha_1,\dots,\alpha_r\in\Omega^1(M)$, $A_1,\dots,A_s\in \spVecX{M}{}$ and $X_1,\dots,X_p\in \spVecX{\partial M}{}$ one has
\begin{eqnarray}
& & \textrm{ptr}(\sigma)(\alpha_1,\dots,\alpha_r,A_1,\dots,A_s,X_1,\dots,X_p)\nonumber \\
& & \qquad:= \sigma(\alpha_1\circ i,\dots,\alpha_r\circ i,A_1\circ i,\dots,A_s\circ i, i_*X_1,\dots,i_*X_{p})\nonumber\,,
\end{eqnarray}
where $i_*:\spVecX{\partial M}{} \to \spVecX{M}{}$ is the pushforward of the inclusion map.

Applying this operator to the stress tensor, we obtain the generalised tensor-valued form $\textrm{ptr}(\cl{T}) \in \Omega^1(M) \tens_i \Omega^{n-1}(\partial M)$, characterised by
$$\textrm{ptr}(\cl{T})(A,X_1,\dots,X_{n-1}) := \cl{T}(A \circ i,i_*X_1,\dots,i_*X_{n-1})$$
for every $A\in \spVecX{M}{}$ and $X_1,\dots,X_{n-1} \in \spVecX{\partial M}{}$.

Notice that the last $n-1$ slots of $\textrm{ptr}(\cl{T})$ can be filled with vectors living on $\partial M$, i.e. represent the "form part" of the tensor, encoding the $(n-1)$-dimensional surface over which $\textrm{ptr}(\cl{T})$ can be integrated. Instead its first slot, the "covector-valued" part of the tensor, can be filled by vectors on $M$, but restricted to those spanning from tangent spaces at points on $\partial M$.
\begin{remark}
This definition is crucial also to technically define the traction force in this differential geometric framework. In fact, even in the case $M$ is the subset of an Euclidian space (where technically the integration of a covector valued form can be computed), the expression $\int_{\partial M} \cl{T}$ for the traction force over the surface $\partial M$ is ill-defined, since $\cl{T}$ is a form over $M$, and not over $\partial M$. Instead the expression $\int_{\partial M} \textrm{ptr}(\cl{T})$ for the traction force is (only in flat spaces) well-defined. 
\end{remark}

\subsection{A new pairing}
We now define a product $\dot\wedge$ between generalised tensor-valued forms that is universally useful in the context of the port-Hamiltonian description of distributed systems. Since, for what concerns the definition of the boundary port, it will only be needed between a generalised covector-valued form and a generalised vector-valued form on a nested manifold $(U,M,f)$, i.e., between 
$$\sigma\in \spKForm{1}{M} \otimes_f\Omega^p(U)\qquad\textrm{and}\qquad\nu\in \spVecX{M}{}\otimes_f\Omega^q(U)\,,$$
we restrict the definition to this case in order to keep the definition technically as simple as possible and define the value contracting wedge product between these as
\begin{eqnarray}
& &(\sigma \dot\wedge \nu)(X^{(1)},\dots,X^{(p+q)})\nonumber\\
& & \qquad :=  \sum_{\pi\in S^{p+q}} \frac{\textrm{sgn}(\pi)}{(p+q)!}\, \sigma\!\left(\nu(X^{(\pi(p+1))},\dots,X^{(\pi(p+q))}),X^{(\pi(1))},\dots,X^{(\pi(p))}\right)\nonumber\,,
\end{eqnarray}
where $S^n$ denotes the set of all permutations of the first $n$ non-zero integers. Clearly, the value-contracting wedge product is an ordinary $(p+q)$-form on $U$:
$$\sigma\dot\wedge\nu \in \Omega^{p+q}(U)\,.$$

Now let us further specialise of the arguments of the $\dot{\wedge}$ pairing in order to define a novel boundary port for the representation of stress power in continuum mechanics. Since the velocity vector field $v$ of the continuum can be identified as a vector-valued zero-form form on $M$ (or equivalently as a tensor-valued form of valence $(1,0,0)$), the following identities hold:
\begin{equation}
\label{eq:extendedHirani}
    \textrm{tr}(\iota_v \cl{T})=\textrm{tr}(\cl{T}\dot{\wedge} v)=\textrm{ptr}(\cl{T})\dot\wedge\textrm{ptr}(v),
\end{equation}
where the $\dot\wedge$ on the middle term is thus defined on the trivially nested manifold $(M,M,\textrm{id}_M)$, whereas the $\dot\wedge$ on the right hand side is defined on the nested manifold $(\partial M, M,i)$, which \textit{de facto} defines the new boundary port.
In fact the the stress power at $\partial M$, at time $t$ is
$$\int_{\partial M} \textrm{tr}(\iota_v \cl{T}) = \langle e|f\rangle,$$
where
$$\langle e|f\rangle := \int_{\partial M} \textrm{ptr}(\cl{T})\dot\wedge\textrm{ptr}(v)$$
is the duality product for the effort 
$e := \textrm{ptr}(\cl{T})$
and the 
flow
$f := \textrm{ptr}(v) \in \spVecX{M}{} \tens_i \Omega^0(\partial M)\,.$ Notice that the flow variable $f$ is just the vector field $v$, but restricted at the boundary, i.e., a map from $\partial M$ to $TM$ (see Fig. \ref{fig_commutative_diagram}), which is indeed the vector that, by definition of partial trace, can be inserted in the first slot of $\textrm{ptr}(\cl{T})$.

\begin{remark}
The introduced notation is slightly different than the one presented in \cite{califano2021geometric,rashad2021exterior}, where the newly defined effort and flow variables are equivalently defined as \textit{two-point tensors}. In particular the notation $\iota^*_2 (\cl{T})$ was used in \cite{califano2021geometric,rashad2021exterior} for the partial trace of the stress, recalling the fact that the pullback applied only on the "second leg" of the stress tensor, and $\iota^*_2(v)$ was used for $\textrm{ptr}(v)$. 
\end{remark} 

\begin{remark}
The operator $\dot{\wedge}$ in the literature, see \cite{kanso2007geometric,gilbert2019geometric,califano2021geometric} for details, is normally introduced in the definition of the \textit{exterior covariant derivative}, an operator which is essential to represent momentum conservation (like e.g., Navier Stokes equations) on manifolds. It is often defined implicitly in the context of two-point tensors, as the binary operator taking as argument vector-valued forms with dual properties on the first leg, which are paired producing a function, while the form parts of the tensors (the second leg) are wedged in the usual sense. Notation-wise, we cannot resist in observing the interesting fact that what is needed to extend the geometric pH structure introduced in \cite{VanDerSchaft2002} and based on the $\wedge$, is an operator which was indicated is $\dot{\wedge}$ \cite{kanso2007geometric} independently of any port-based structure.
\end{remark}

\begin{figure}
\centering
\includegraphics[width = 0.5\textwidth]{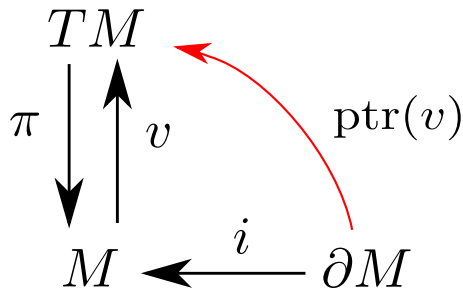}
\caption{Commutative diagram illustrating the definition of the partial trace of a vector field $v\in \spVecM$.}
\label{fig_commutative_diagram}
\end{figure}

\subsection{A new representation of the moving domain port}
The previous construction introduced to represent boundary power flows in continuous mechanics has an important consequence: it induces an alternative way to represent the moving domain port (\ref{eq:boundaruEffort1}-\ref{eq:boundaryPairing1}), which will be fundamental to impose the physically relevant boundary condition in a FSI system, i.e., no slip.

The key observation that explains why this is true, is that the moving domain port is always characterised by an effort being the energy density $\cl{H}$ of the system. Being $\cl{H}$ a top-form, it always possesses a representation as a covector-valued form, i.e., the energy density and stress are tensors of the same valence. As a consequence, the construction made for stress, and in particular identities in (\ref{eq:extendedHirani}), will hold also for the moving domain port, with $\cl{H}$ in lieu of $\cl{T}$.

In summary, in contrast to (\ref{eq:boundaruEffort1}-\ref{eq:boundaryPairing1}), the new representation for the moving boundary port $(e_u,f_u)$ is then defined as
\begin{align}
\label{eq:boundaruEffort2}
    e_u=\textrm{ptr}(\mathcal{H})\in \Omega^1(M_t) \tens_i \Omega^{n-1}(\partial M_t),\\
    \label{eq:boundaruFlow2}
    f_u=\textrm{ptr}(u) \in \spVecX{M_t}{} \tens_i \Omega^0(\partial M_t),\\
    \label{eq:boundaryPairing2}
    \langle e_u|f_u \rangle :=\int_{\partial M_t}e_u \dot{\wedge} f_u.
\end{align}

Even if the power computed by means of the two pairings produces the same result, the second version of the port allows for representing more information on the single effort and flow variables. This follows by noting that in the first representation $f_u=\textrm{tr}(\star u^{\flat})$ is a top form on the manifold $\partial M$, which encodes the information only on the normal component of the vector field $u$ evaluated at points in $\partial M$. In fact, if $u$ is tangent to $\partial M$, then $\textrm{tr}(\star u^{\flat})=0$. 
This makes sense since in this case the pairing represents pure energy advection, and the effort variable $e_u=\textrm{tr}(\star \mathcal{H})$ is just the evaluation at $\partial M$ of the energy function. As a consequence only normal components of the velocity field do influence the total energy content in $M$. In the second representation instead, the flow variable $\textrm{ptr}(u)$ contains the information on the complete vector field $u$, once restricted to $\partial M$.

It is intuitively clear, and will be formalised in the sequel, that this difference will matter in the moment that specific \textit{boundary conditions} have to be imposed on two adiacent continua: for a \textit{no-penetrability} condition, the first port representation characterised by the $\wedge$ pairing is sufficiently expressive since only normal components of velocity field play a role; instead when a \textit{no-slip} condition needs to be imposed, the second representation of the pairing is needed since the condition is on the whole vector field at the boundary of the two continua, and not just on its normal component.
This aspect has been poorly addressed in the geometric port Hamiltonian literature since the presented models were based on "standard" Stokes Dirac Structure in the sense of \cite{VanDerSchaft2002}, where only scalar valued forms were considered, and the pullback of the inclusion map operation was normally hidden since in that case it was trivially distributed on the wedge product. 

\subsection{Power balance in Newtonian fluids and Eulerian/Lagrangian description}

Now we address the importance of the introduced pairing in the expression of power flows due to stress and energy advection in viscous fluid dynamics, which will be needed to present the FSI model. We refer to \cite{califano2021geometric} for a detailed explanation of the presented fluid dynamic model.

In fluid dynamics the stress is normally decomposed into the sum of hydrostatic pressure and viscous stress as
\begin{equation}
\label{eq:totalStress}
    \cl{T}=-\star p + \cl{T}_{\mathrm{v}} \in \spKFormM{1}\tens\spKFormM{n-1}
\end{equation}
where $\cl{T}_{\mathrm{v}}$ is the viscous stress tensor and $p\in \spFn{M}$ is the hydrostatic pressure function.
When considering Newtonian fluids, $\cl{T}_{\mathrm{v}}$ is defined as the sum of a bulk stress $\cl{T}_{\lambda}$ and a shear stress $\cl{T}_{\kappa}$, defined by \cite{gilbert2019geometric,califano2021geometric}
\begin{align}
    \cl{T}_{\lambda} &:= \lambda (\star \textrm{div}(v)) ,
    \label{eq:DefBulkViscosity}\\
   \cl{T}_{\kappa} &:= 
\kappa (\star_2 \cl{L}_{v}g), \label{eq:DefShearViscosity}
\end{align}
where $\lambda$ and $\kappa$ are respectively the bulk and shear viscosity coefficients.
The bulk stress $\cl{T}_{\lambda}$ depends on the divergence of the velocity vector field $\textrm{div}(v)\in \spFn{M}$, defined as the function such that $\cl{L}_{v}\volF=\textrm{div}(v)\volF$ holds true.
The shear stress $\cl{T}_{\kappa}$ is defined in order to model viscous stresses whenever the transport of the metric under the flow of $v$ is non-zero, i.e., when $v$ fails to be the generator of a rigid body motion. In fact $\cl{L}_v g$ extends the concept of \textit{rate of strain}
to Riemannian manifolds. For example, its components in a Euclidean space on a Cartesian chart are $(\cl{L}_v g)_{ij}=\partial_j v^i+\partial_i v^j$, clearly resembling the standard vector calculus definition of rate of strain \ramy{in} Eucledian space. 

It is important to notice that, contrarily to the bulk stress \ramy{$\cl{T}_{\lambda}$} and the hydrostatic pressure density $\star p$, the shear stress \ramy{$\cl{T}_{\kappa}$} does not admit a formulation as a scalar valued differential form. In fact, being $\cl{L}_v g$ a symmetric 2-rank tensor, it cannot be represented by a scalar valued differential form, which is by definition a totally antisymmetric tensor field. In other words, the difference between a covector-valued $(n-1)$-form \ramy{($\cl{T}_\kappa$)} and a top form \ramy{($\cl{T}_{\lambda},\star p, \cl{H}$)}, is that the former is antisymmetric in its last $n-1$ slots, while the latter in all of them. This is ultimately the phenomenological reason why a geometric representation of Navier--Stokes equations needs to be developed using (co)vector--valued forms \cite{frankel2011geometry}, and why e.g., \textit{Euler equations} can be represented with standard differential forms only.

\ramy{As a consequence,} the stress power \ramy{(\ref{eq:SurfacePowerStress})} for the pressure \ramy{$\star p$} and the bulk stress \ramy{$\cl{T}_\lambda$ parts in (\ref{eq:totalStress})} can be expressed with \ramy{both} the $\wedge$ pairing and the $\dot{\wedge}$ pairing, exactly in the same way it was discussed for the energy density $\cl{H}$.

In \cite{califano2021geometric} the port-Hamiltonian model for Navier-Stokes equation is given on a fixed spatial domain (\ramy{i.e. }$u=0$) and the model produces an instance of (\ref{eq:continuityFixedBoundary}) with boundary flux terms $\textrm{tr}(\Phi)=\textrm{tr}(\iota_v(\cl{T}-\cl{H}))$ and source terms $\sigma$ expressing energy dissipation inside the spatial domain as quadratic functions on the rate of strain $\cl{L}_v g$ (shear stress) and on the divergence $\textrm{div}(v)$ (bulk viscosity). While the bulk stress power admits a representation based on the $\wedge$ pairing (since $\cl{T}_{\lambda}\dot{\wedge} v= (\star \cl{T}_{\lambda}) \wedge (\star v^{\flat})$ ), the shear stress can only be represented within a boundary power port using the $\dot{\wedge}$ pairing. As a matter of fact, using the notation introduced in this paper, in \cite{califano2021geometric} it is shown that

\begin{equation}
\dot{H}=\int_{\partial M} \underbrace{\textrm{ptr}(\cl{T}-\cl{H})}_{e_\Phi} \dot{\wedge} \underbrace{\textrm{ptr}(v)}_{f_\Phi}+\underbrace{\int_M \sigma}_{\textrm{Dissipation}}.
\end{equation}
The power flux expression $\Phi \in \Omega^{n-1}(M)$ is, on any Riemannian manifold, as simple as
\begin{equation}
    \Phi=  \iota_v(\cl{T}-\cl{H})
\end{equation}
where the $\cl{T}$ term is responsible for capturing the power flux due to stress while the $\cl{H}$ part captures the power flux due to pure energy advection. We will refer in the sequel to the object $\cl{T}-\cl{H}$ as the \textit{energy-stress tensor} of the continuum.

The power flux due to energy advection is a consequence of the fact that a fixed domain $M$ is considered as spatial container of the fluid, whose boundary is amenable to energy transfer. In classical fluid dynamic terminology, this corresponds to an \textit{Eulerian} description of the fluid. Allowing for a representation of the fluid dynamic system on a moving spatial domain $M_t$ with instantaneous velocity vector field $u$ as described previously, one obtains the relative power flux
\begin{equation}
    \Phi_{\textrm{rel}}= \iota_v(\cl{T}-\cl{H})+ \iota_u(\cl{H}).
\end{equation}
A representation in which the domain moves with the macroscopic velocity of the fluid is normally called a \textit{Lagrangian} representation, and in this context this corresponds to imposing $u=v$. In this case the physical relative power flux becomes 
\begin{equation}
\label{eq:relativeFlux}
    \Phi_{\textrm{rel}}= \iota_v(\cl{T})
\end{equation}
showing that at a power balance level the effects due to advection are rightfully factored out due to the specific representation that is used, and only stress effects influence the boundary power flow of the system. 

In the following section we show how to put together all the discussed insights to represent a physical fluid-structure interaction model as an example of a physically meaningful application which exhibits a varying boundary. It is worth noticing that the presented framework can be used also to model varying boundaries within a fluid for any relative motions (any $u$) whose extreme cases could be the Eulerian ($u=0$) or Lagrangian ($u=v$) ones.

\section{Port-Hamiltonian formulation of Fluid-Structure Interaction (FSI)}\label{sec:pH_fsi}
\label{sec:fsi}

\begin{figure}
\centering
\includegraphics[width = 0.7\textwidth]{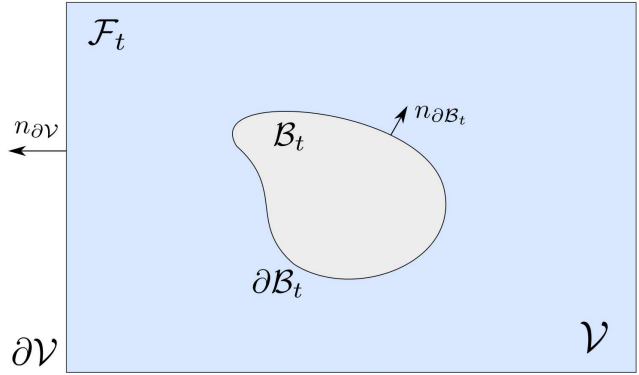}
\caption{Schematic showing the control volume $\cl{V}$ containing the fluid domain $\cl{F}_t$ and the rigid body domain $\cl{B}_t$ at a given time instant $t$. We follow the convention that the submanifolds $\partial \cl{V}$ and $\partial \cl{B}_t$ have outwards orientation.}
\label{fig_fsi_scehmatic}
\end{figure}

In this section we describe the geometric port-Hamiltonian FSI model as an application of distributed systems with moving spatial domain.

\ramy{With reference} to Fig. \ref{fig_fsi_scehmatic}, we denote by $\cl{V} \subset \mathbb{R}^n$ the "control volume", containing the fluid and the rigid body. Being a subset of $\mathbb{R}^n$, $\cl{V}$ inherits the standard Euclidean metric and volume form. Of course\stefano{(non relativistic)} physically meaningful flows are represented by the condition $n \in \{2,3\}$. The closed and connected subset of $\cl{V}$ containing the rigid body at time $t$ is denoted $\cl{B}_t:=\cl{B}(t) \subset \cl{V}$. Its complement $\cl{F}_t:=\cl{V} \setminus \cl{B}_t$ represents the subset of $\cl{V}$ containing the fluid at time $t$. The $(n-1)$-dimensional manifolds corresponding to the boundary of the control volume and of the rigid body are denoted respectively $\partial \cl{V}$ and $\partial \cl{B}_t$. As a consequence $\partial \cl{F}_t=\partial \cl{V} \cup \partial \cl{B}_t$, \stefano{or better, as chains $\partial \cl{F}_t=\partial \cl{V} + \partial \cl{B}_t$}
where $\partial \cl{F}_t$ and $\partial \cl{B}_t$ exhibit opposite orientations\stefano{seen as boundary of $\cl{F}_t$}.

We define the two inclusion maps $i_b: \partial \cl{B}_t \to \cl{B}_t \subset \cl{V}$ and $i_f: \partial \cl{F}_t \to \cl{F}_t \subset \cl{V}$, mapping points of the respective boundaries to the common ambient space represented by the control volume $\cl{V}$. For the partial trace, we will use the notation $\text{ptr}_{I}(\cdot)$, with $I \in \{b,f\}$ to specify with respect to which inclusion map the operation is performed. We assume that $\cl{B}_t$ is a differentiable 
surface ($n=3$) or curve ($n=2$). While the fluid particles are allowed to leave and enter the control volume $\cl{V}$, we assume that the particles of the body (i.e. all $q \in \cl{B}_t$) remain within $\cl{V}$ for all $t$, which avoids \ramy{changing the} topology of $\cl{F}_t$.

In the classical Hamiltonian approach, see e.g. \ramy{\cite{jacobs2012fluid,glass2012movement,fusca2018groupoid}}, the dynamic equations governing the "fluid plus rigid body" system are derived using reduction or variational techniques starting from the configuration of the whole system. Such approach implies that the fluid particles remain always confined within $\cl{V}$, and as such require that the overall system is indeed an isolated, closed system. This is exactly the limitation that the pH framework aims at overcoming, since we will derive the FSI dynamics by modeling the fluid and the rigid body separately as open dynamical systems, which are then interconnected together in an energy-preserving way. Such interconnection will capture the energy exchanged between the rigid body and the fluid through the moving boundary $\cl{B}_t$ as well as the energy exchange between the fluid and the exterior world (with respect to $\cl{V}$) through the fixed boundary $\partial \cl{V}$. It is worth noticing that the pH approach allows easily incorporating a time varying boundary $\partial \cl{V}$ for the control volume. However for what follows we treat $\cl{V}$ as the fixed "ambient space" used to describe the dynamics on $\cl{F}_t$ and $\cl{B}_t$ embedded in $\cl{V}$. Even if the presented methodology can be applied to any kind of flow, we choose to focus on the incompressible case since it provides a neat comparison with low-speed aerodynamic approaches, which mainly consider incompressible flows in FSI systems. Furthermore it does not constitute a loss of generality in the application of the pH approach of interconnection of systems to generate a FSI model, which is the main goal of this section.
\subsection{Port-Hamiltonian model of Incompressible Viscous Flow on a Moving Domain}

\begin{figure}
\centering
\includegraphics[width = 0.8\textwidth]{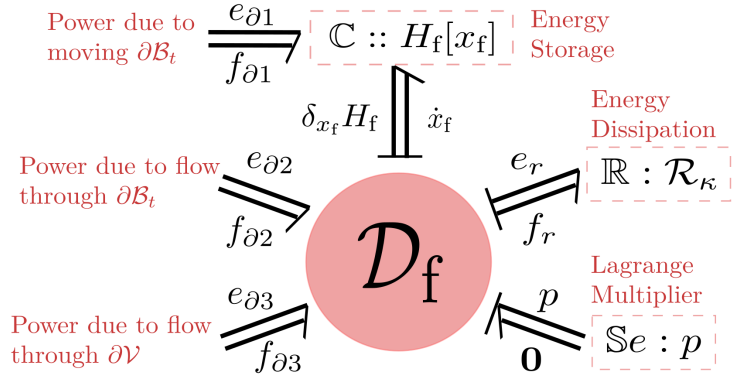}
\caption{Bond-graph representation of the incompressible viscous flow port-Hamiltonian model \ramy{as a network of interconnected energetic subsystems}.}
\label{fig_pH_fluid_bond_graph}
\end{figure}

\begin{table}
	\centering
		\renewcommand{\arraystretch}{1.3}
	\begin{tabular}{l | l }
		Variable 			  							& Description\\ \hline
		$\dot{x}_{\mathrm{f}}=(\dot{\tilde{v}},\dot{\mu})\in \spKForm{1}{\cl{F}_t} \times \spKForm{n}{\cl{F}_t}$	& rate of change of state variables \\
		$\delta_{x_{\mathrm{f}}} H_\text{f} = (\delta_{\tilde{v}}H_\text{f},\delta_{\mu}H_\text{f})\in\spKForm{n-1}{\cl{F}_t}\times \spKForm{0}{\cl{F}_t}$	& co-energy variables \\
		$e_r=\cl{T}_{\kappa} \in \spKForm{1}{\cl{F}_t} \tens \spKForm{n-1}{\cl{F}_t}$ & shear stress tensor \\
		$f_r=\nabla v \in \spVecX{\cl{F}_t}{} \tens \spKForm{1}{\cl{F}_t}$ & velocity gradient tensor \\
		$p\in\spKForm{0}{\cl{F}_t}$ 	& static pressure function \\
		
		$e_{\partial 1} = -\text{ptr}_{b}(\cl{H}_{\text{f}}) \in \spKForm{1}{ \cl{B}_t}\tens_{i_b} \spKForm{n-1}{\partial \cl{B}_t}$ & energy density on $\partial \cl{B}_t$ \\
		$f_{\partial 1} = \textrm{ptr}_{b}(u) \in \spVecX{ \cl{B}_t}{} \tens_{i_b} \spKForm{0}{\partial \cl{B}_t}$ & velocity of $\partial \cl{B}_t$ (moving domain)\\
		$e_{\partial 2} = -\text{ptr}_b(\cl{T}-\cl{H}_\text{f}) \in \spKForm{1}{ \cl{B}_t}\tens_{i_b} \spKForm{n-1}{\partial \cl{B}_t}$ & energy-stress tensor on $\partial \cl{B}_t$ \\
		$f_{\partial 2}= \textrm{ptr}_{b}(v) \in \spVecX{\cl{B}_t}{} \tens_{i_b} \spKForm{0}{\partial \cl{B}_t}$ & velocity of the fluid on $\partial \cl{B}_t$ \\
		$e_{\partial 3} = \text{ptr}_{f}(\cl{T}-\cl{H}_\text{f}) \in \spKForm{1}{ \cl{V}}\tens_{i_f} \spKForm{n-1}{\partial \cl{V}}$ & energy-stress tensor on $\partial \cl{V}$ \\
		$f_{\partial 3} = \textrm{ptr}_{f}(v) \in \spVecX{\cl{V}}{} \tens_{i_f} \spKForm{0}{\partial \cl{V}}$ & velocity of the fluid on $\partial \cl{V}$ \\ \hline
	
	\end{tabular}
	
	\caption{Port variables of the \pH model of incompressible viscous fluid}
	\label{tab_pH_port_variables}
\end{table}

As presented in \cite{califano2021geometric,rashad2021exterior}, in the pH framework, the model for an incompressible viscous flow is composed by several energetic subsystems, connected by power ports as depicted in Fig. \ref{fig_pH_fluid_bond_graph}. Table \ref{tab_pH_port_variables} summarises the spaces of effort and flow variables characterising these subsystems, described in the following. The total energy of the system is represented by the $\mathbb{C}$-element representing kinetic energy stored in the spatial domain $\cl{F}_t$, given by the Hamiltonian functional 
\begin{equation}
\label{eq:HInc}
    H_\text{f}[\tilde{v},\mu]=\int_{\cl{F}_t}\cl{H}_\text{f}(\tilde{v},\mu)=\int_{\cl{F}_t} \frac{1}{2}(\star \mu)\tilde{v} \wedge \star \tilde{v},
\end{equation}
where $\tilde{v}:=\flat (v) \ramy{\in \spKForm{1}{\cl{F}_t}}$ is the one-form obtained from the fluid's velocity vector field $v \in \spVecX{\cl{F}_t}{}$, while $\mu:= \star \rho \in \spKForm{n}{\cl{F}_t}$ 
is the mass top-form, defined as the Hodge dual of the mass density function of the fluid $\rho \in \Omega^0(\cl{F}_t)$. We denote the fluid state by $x_{\mathrm{f}}:=(\tilde{v},\mu)$. The rate of change of stored kinetic energy is expressed as
\begin{equation}
\label{eq:HdotInc}
    \dot{H}_\text{f}=\int_{\cl{F}_t} \delta_{x_{\mathrm{f}}}H_\text{f} \wedge \dot{x}_{\mathrm{f}}+\int_{ \partial \cl{B}_t} e_{\partial 1} \dot{\wedge} f_{\partial 1},
\end{equation}
where the rate of change of the state variables $\dot{x}_{\mathrm{f}}=(\dot{\tilde{v}},\dot{\mu})$ and the variational derivatives $\delta_{x_{\mathrm{f}}}H_\text{f}=(\delta_{\tilde{v}}H_\text{f},\delta_{\mu}H_\text{f})$ represent the flow and effort variables of the energy storage subsystem, respectively. The effort variables of the storage (also called co-energy variables) are given by \cite{rashad2020porthamiltonian1,rashad2020porthamiltonian2}:
\[ \delta_{\tilde{v}}H_\text{f}=(\star \mu) \star \tilde{v} \in \Omega^{n-1}(\cl{F}_t), \qquad \qquad  \delta_{\mu}H_\text{f}=\frac{1}{2}\iota_{v}\tilde{v} \in \Omega^0(\cl{F}_t),\]
representing respectively the mass flow flux-form and the dynamic pressure function \ramy{(modulo the density since $\iota_{v}\tilde{v} = \tilde{v}(v) = g(v,v) $ )}. Furthermore, the boundary port $(e_{\partial 1}, f_{\partial 1})$ characterises the power due to variation of the boundary $\partial \cl{B}_t$ as described in (\ref{eq:boundaruEffort2}-\ref{eq:boundaryPairing2}). As a matter of fact notice that (\ref{eq:HdotInc}) is the specialisation for incompressible fluids of the general power balance for systems with moving domains described in (\ref{eq:DensityVariationMovingBoundary2}), where the moving domain is expressed by the variation of $\partial \cl{B}_t$ (moving with velocity $u$) only since we assumed $\partial \cl{V}$ to be fixed.
Thus, we have
\begin{align}
\label{eq:partial1e}
    e_{\partial 1}&=-\textrm{ptr}_{b}(\cl{H}_\text{f}), \\
    \label{eq:partial1f}
    f_{\partial 1}&=\textrm{ptr}_{b}(u),
\end{align}
where the minus sign in the definition of $e_{\partial 1}$, in contrast to (\ref{eq:boundaruEffort2}), is due to the opposite orientation of $\partial \cl{B}_t$ with respect to $\partial \cl{F}_t$ as described in Fig. \ref{fig_fsi_scehmatic}.

The second energetic subsystem of the fluid's pH model in Fig. \ref{fig_pH_fluid_bond_graph} is the $\mathbb{R}$-element characterising the dissipation occurring within the spatial domain $\cl{F}_t$ due to shear viscosity. As detailed in \cite{califano2021geometric}, the effort and flow variables $(e_r,f_r)=(\cl{T}_{\kappa}, \nabla v)$ are given by the Cauchy shear stress and the velocity gradient, related by Stokes' constitutive relation represented by $e_r=\cl{R}_{\kappa}(f_r)$ such that
\begin{equation}
\label{eq:resistive}
    \int_{\cl{F}_t} e_r \dot{\wedge} f_r =\int_{\cl{F}_t} \cl{R}_{\kappa}(f_r) \dot{\wedge} f_r \geq 0,
\end{equation}
where the non decreasing property of the constitutive relation $\cl{R}_{\kappa}$ characterised the irreversible energy transfer to the thermal domain. The reader is referred to \cite{califano2021geometric} for more details on this representation of energy dissipation due to shear viscosity.

Along with the energy storage and dissipation subsystems, the fluid's pH model consists of four ports which are open for interconnection, one of which is $(e_{\partial 1}, f_{\partial 1})$, representing power transfer due to variation on $\partial \cl{B}_t$ as explained previously.
The other boundary ports, namely $(e_{\partial 2}, f_{\partial 2})$ and $(e_{\partial 3}, f_{\partial 3})$, describe the power exchange due to mass flow through the fluid's boundaries $\partial \cl{B}_t$ and $\partial \cl{V}$, respectively. The corresponding effort and flow variables are given by
\begin{align}
\label{eq:partial2e}
    e_{\partial 2}&=-\text{ptr}_{b}(\cl{T}-\cl{H}_{\text{f}}), \\
    \label{eq:partial2f}
    f_{\partial 2}&=\textrm{ptr}_{b}(v), \\
    \label{eq:partial3e}
     e_{\partial 3}&=\text{ptr}_{f}(\cl{T}-\cl{H}_{\text{f}}), \\
     \label{eq:partial3f}
    f_{\partial 3}&=\textrm{ptr}_{f}(v),
\end{align}
with the sign difference between $e_{\partial 2}$ and $e_{\partial 3}$ reflecting the difference in orientation between $\partial \cl{B}_t$ and $\partial \cl{V}$. Here $\cl{T}$ is the total stress given in (\ref{eq:totalStress}), where the bulk part of the stress is identically zero due to \ramy{the} incompressibility \ramy{of the fluid flow}.
The last open port of the pH model is the distributed port $(p,0)\in \Omega^0(\cl{F}_t) \times \Omega^n(\cl{F}_t)$, representing the action of the fluid's static pressure $p$, a Lagrangian multiplier that enforces the incompressibility condition $\extd \star \tilde{v}=0$, without exchanging power with the fluid, i.e. there is zero power flowing through the port $(p,0)$. Notice that this Lagrangian multiplier completely characterises the static pressure in the tensor $\cl{T}$, which is indeed not dependent on any thermodynamic potential in the incompressible case \cite{rashad2020porthamiltonian2}.

The final component of the fluid's pH model is the SDS $\cl{D}_\text{f}$, which is mathematically an infinite-dimensional subspace of the total port-space $\mathfrak{F}_\text{f} \times \mathfrak{E}_\text{f}$ corresponding to the four ports $(\delta_{x_{\mathrm{f}}} H_\text{f}, \dot{x}_{\mathrm{f}}),(e_r,f_r),(e_{\partial 2}, f_{\partial 2})$, and $(e_{\partial 3}, f_{\partial 3})$, with $\mathfrak{F}_\text{f}$ denoting the total space of flows and $\mathfrak{E}_\text{f}:=\mathfrak{F}_\text{f}^*$ the total space of efforts, i.e., $(\dot{x}_{\mathrm{f}},f_r,f_{\partial 2},f_{\partial 3})\in \mathfrak{F}_\text{f}$ and $(\delta_{x_{\mathrm{f}}}H_\text{f},e_r,e_{\partial 2},e_{\partial 3})\in \mathfrak{E}_\text{f}$, where the individual spaces are given in Table \ref{tab_pH_port_variables}.

The explicit representation of the pH model consisting of all the components explained above is given by
\begin{align}
	\dot{\tilde{v}} =& - \extd (\delta_{\mu} H_\text{f}) - \iota_v \extd \vf - \frac{\extd p}{\star \mu} + \frac{\kappa}{\star \mu} \star \extd \star \extd \vf, \label{eq:PDE_1}\\
	\dot{\mu} =& - \extd(\delta_{\tilde{v}}H_\text{f}), \label{eq:PDE_2}\\
	0 =& \extd \star \vf. \label{eq:PDE_3}
\end{align}
which are the incompressible Navier-Stokes equations in covariant form.
The reader is referred to \cite{califano2021geometric} for a detailed description of the port-Hamiltonian structure of viscous fluid dynamics which underlies these equations \ramy{and to \cite{rashad2021exterior} for their vector calculus representations}. 

Furthermore, $\cl{D}_\text{f}$ encodes the power balance
\begin{equation}
    -\int_{\cl{F}_t} (\delta_{x_{\mathrm{f}}}H_\text{f} \wedge \dot{x}_{\mathrm{f}}-e_r \dot{\wedge} f_r)+\int_{\partial \cl{B}_t}e_{\partial 2} \dot{\wedge} e_{\partial 2}+\int_{\partial \cl{V}}e_{\partial 3} \dot{\wedge} e_{\partial 3}=0
\end{equation}
which, using (\ref{eq:partial1e}-\ref{eq:partial3f}) can be rewritten as
\begin{align}
    \dot{H}_\text{f}&=-\int_{\cl{F}_t} e_r \dot{\wedge} f_r+\int_{\partial \cl{B}_t}(e_{\partial 1} \dot{\wedge} e_{\partial 1}+e_{\partial 2} \dot{\wedge} e_{\partial 2})+\int_{\partial \cl{V}}e_{\partial 3} \dot{\wedge} e_{\partial 3} \nonumber \\
    & \leq \int_{\partial \cl{B}_t}[\text{ptr}_{b} (\cl{H}_\text{f}-\cl{T})\dot{\wedge} \textrm{ptr}_{b}(v)-\text{ptr}_{b} (\cl{H}_\text{f})\dot{\wedge} \textrm{ptr}_{b}(u)]+\int_{\partial \cl{V}} \text{ptr}_{f} (\cl{H}_\text{f}-\cl{T})\dot{\wedge} \textrm{ptr}_{f}(v)
\end{align}
which states that the rate of change in kinetic energy in the fluid domain $\cl{F}_t$ is equal to the external supplied power through the moving boundary $\partial \cl{B}_t$ and the fixed boundary $\partial \cl{V}$, in addition to the internally dissipated power due to viscosity, leading to the inequality on $\dot{H}_\text{f}$.

\ramy{\begin{remark}
Note that in Fig. \ref{fig_pH_fluid_bond_graph} the storage port $(\delta_{x_{\mathrm{f}}}H_\text{f},\dot{x}_{\mathrm{f}})$, the resistive port $(e_r,f_r)$ and the Lagrange multiplier port $(p,0)$ have a specified causality, indicated by the stroke on their respective ports.
On the other hand, the three boundary ports $(e_{\partial i},f_{\partial i}), i\in\{1,2,3\},$ do not have a specific causality as it depends on the external interconnection.
For the storage port, the Dirac structure specifies the flow $\dot{x}_{\mathrm{f}}$ to the $\bb{C}$-element which returns information about the effort $\delta_{x_{\mathrm{f}}}H_\text{f}({x}_{\mathrm{f}})$ after integration in time.
Similarly, the Dirac structure specifies the flow $f_r$ to the $\bb{R}$-element which returns the effort $e_r = \cl{R}_\kappa(f_r)$.
As for the Lagrange multiplier port, the effort $p$ should specified to the Dirac structure such that its dual flow variable is always zero.
\end{remark}}

\subsection{Port-Hamiltonian model of the Rigid Body}

Now we turn attention to the pH model for a generic rigid body motion in $\mathbb{R}^n$ in the presence of a gravitational field.
\ramy{While all the following constructions could be presented in a coordinate-free and dimension-independent manner as in \cite{Stramigioli2001,van1997interconnected}, we will present it next for $n=3$ and using matrix representations for ease of exposition and to make the subject accessible to a wider audience. For that purpose, we introduce the body-fixed reference frame $\PsiB$, attached to the center of mass of the rigid body at an arbitrary orientation, and the inertial reference frame $\PsiV$ fixed at an arbitrary point in $\cl{V}$.}
In what follows we detail the \pH model of the rigid body depicted in Fig. \ref{fig_pH_body_bond_graph} while the spaces of effort and flow variables of the different components of the model are summarised in Table \ref{tab_pH_port_variables_rb}. 

\begin{figure}
\centering
\begin{subfigure}{0.9\textwidth}
         \centering
         \includegraphics[width=\textwidth]{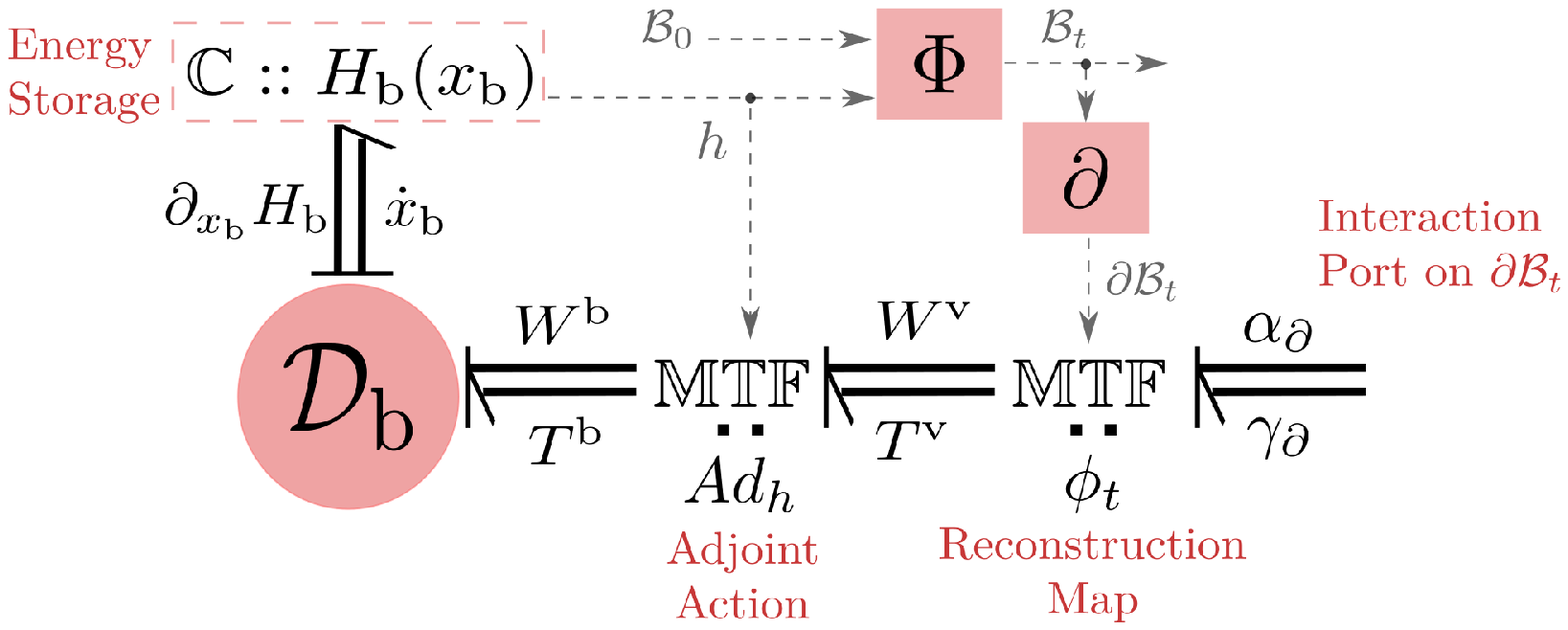}
         \caption{Bond graph}
\end{subfigure}
\\
\begin{subfigure}{0.8\textwidth}
         \centering
         \includegraphics[width=\textwidth]{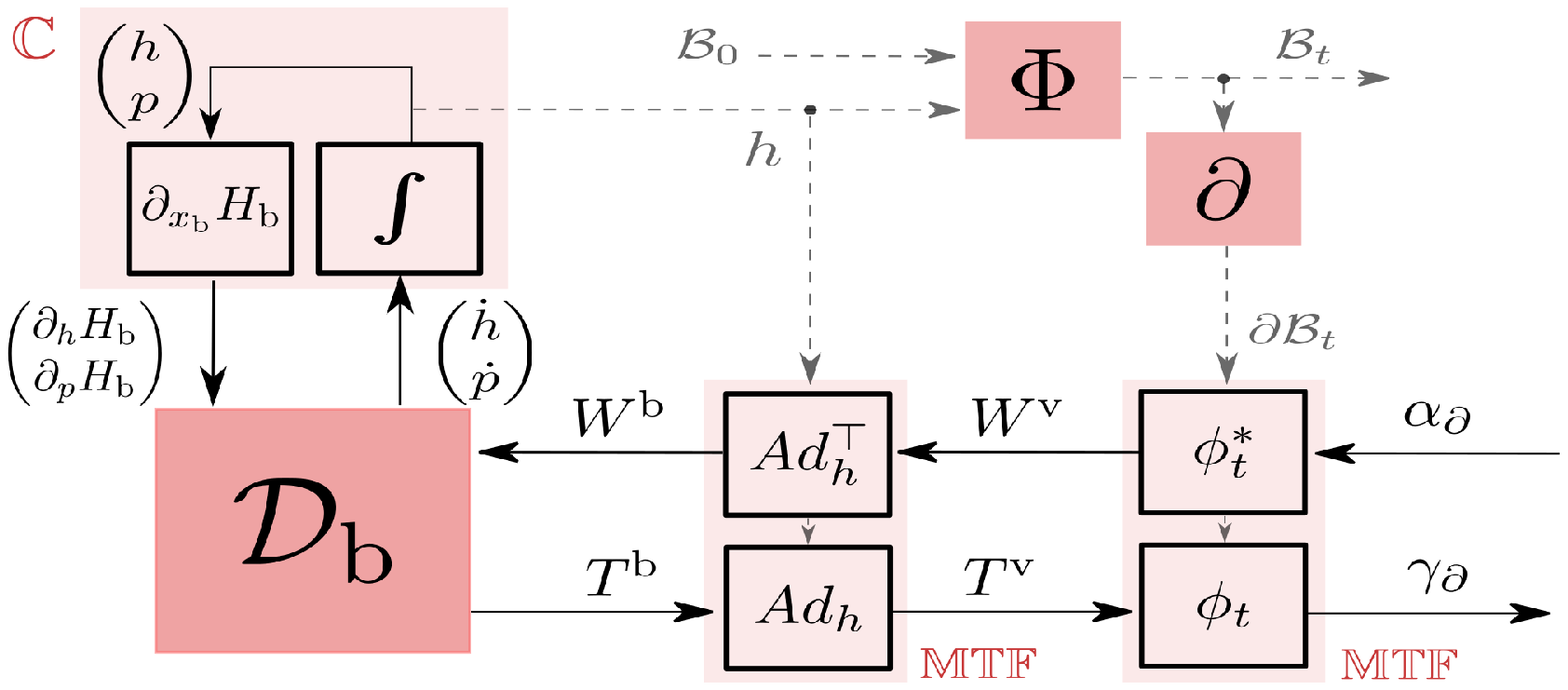}
         \caption{Block diagram}
\end{subfigure}
\caption{Graphical representation of the rigid body port-Hamiltonian model as a bond-graph (a) and its corresponding block-diagram (b).}
\label{fig_pH_body_bond_graph}
\end{figure}

\begin{table}
	\centering
		\renewcommand{\arraystretch}{1.3}
	\begin{tabular}{l | l }
		Variable 			  							& Description\\ \hline
		$\dot{x}_{\mathrm{b}}=(\dot{h},\dot{p}) \in T_h SE(3) \times \mathfrak{se}^*(3)$	& rate of change of state variables \\
		$\partial_{x_{\mathrm{b}}}H_\text{b}=(\partial_{h}H_\text{b},\partial_{p}H_\text{b})\in T_h^*SE(3)\times \mathfrak{se}(3)$	& co-energy variables \\
		$T \in \mathfrak{se}(3)$ & twist \\
		$W \in \mathfrak{se}^*(3)$ 	& wrench \\
		
		$\gamma_{\partial} \in \spVecX{ \cl{B}_t}{} \tens_{i_b} \spKForm{0}{\partial \cl{B}_t}$ & reconstructed twist on $\partial \cl{B}_t$ \\
		$\alpha_{\partial} \in \spKForm{1}{\cl{B}_t} \tens_{i_b} \spKForm{n-1}{\partial \cl{B}_t}$ & reconstructed wrench on $\partial \cl{B}_t$ \\ \hline
	
	\end{tabular}
	
	\caption{Port variables of the \pH model of rigid body motion}
	\label{tab_pH_port_variables_rb}
\end{table}

First, the energy storage $\mathbb{C}$-element characterises the rigid body's total energy composed of kinetic and gravitational potential energy. After some choices of references, and considering in what follows the matrix Lie groups as representation for $SE(3)$ and $SO(3)$ and the corresponding Lie-algebra, the state variable is given by $x_{\mathrm{b}}:=(h,p) \in SE(3) \times \mathfrak{se}^*(3)$ which consists of the configuration and the generalised momentum of the rigid body, respectively. The configuration space is identified with $SE(3)$, the space of positive isometries on $\mathbb{R}^3$, representing proper rigid body motions (i.e., rotations and translations without reflections). It is well known that $SE(3)$ is the semidirect product group of the special orthogonal group $SO(3)$ and $\mathbb{R}^3$. Thus we write $h=(R,\xi)$, with $R\in SO(3)$ \ramy{denoting the orientation of $\PsiB$ with respect to $\PsiV$ and $\xi \in \mathbb{R}^3$ denoting the origin of $\PsiB$ expressed in $\PsiV$}. The group composition operator and inverse of $SE(3)$ are defined by:
\begin{align}
    h_1 \circ h_2 &:=(R_1,\xi_1)\circ (R_2,\xi_2)=(R_1 R_2, \xi_1 + R_1 \xi_2),\\
    h^{-1}&:=(R,\xi)^{-1}=(R^{-1},-R^{-1}\xi).
\end{align}
The space $\mathfrak{se}^*(3)$ is the dual of the Lie algebra $\mathfrak{se}(3)$ of the (matrix) Lie group $SE(3)$, where the elements $T\in \mathfrak{se}(3)$ represent the configuration independent
velocity (i.e., rotational and translational) of the rigid body, which is referred to as twist. The Lie algebra $\mathfrak{se}(3)$ is identifiable with $\mathfrak{so}(3) \times \mathbb{R}^3$ with $\mathfrak{so}(3)$ the Lie algebra of $SO(3)$, which can be identified with $\mathbb{R}^{3}$ using the isomorphism
\begin{align}
    \cl{S}:\mathbb{R}^{3} &\to \mathfrak{so}(3)\nonumber \\
    \omega &\to \cl{S}(\omega)=: \tilde{\omega},
    \label{eq:skewsymIsom}
\end{align}
which takes the form
 \begin{align*}
     \mathbb{R}^3 \ni \begin{pmatrix} \omega^1 \\ \omega^2 \\ \omega^3 \end{pmatrix} \mapsto \begin{pmatrix}
     0 & -\omega^3 & \omega^2 \\ \omega^3 & 0 & -\omega^1 \\ -\omega^2 & \omega^1 & 0 \end{pmatrix} \in \mathfrak{so}(3).
 \end{align*}

Therefore, we can associate to every twist $T \in \mathfrak{se}(3)$ a pair of vectors $(\omega, v) \in \mathbb{R}^{3} \times \mathbb{R}^3$, that represent the angular and linear velocities of the rigid body, respectively.
By duality, we can do the same for elements of $\mathfrak{se}^*(3)$. 
\ramy{As mentioned earlier}, we present all the details of the pH model in this paper using the vector representations of $\mathfrak{se}(3)$ and $\mathfrak{se}^*(3)$ instead of using abstract vector spaces. \ramy{Therefore, we will interchangeably identify elements of $\mathfrak{se}(3)$ and $\mathfrak{se}^*(3)$ with vectors in $\mathbb{R}^6$, which can be represented either in $\PsiB$ or $\PsiV$.} 

The total energy stored by the rigid body is characterised by the Hamiltonian function given by the sum of kinetic and gravitational potential energy:
\begin{equation}\label{eq:Ham_rigid_body}
    H_\text{b}(h,p)=\half p^{T} \cl{I}^{-1} p + m g^{T} \xi,
\end{equation}
\ramy{where the generalized momentum $p$ is expressed in $\PsiB$, $\cl{I} \in \mathbb{R}^{6 \times 6}$ denotes the matrix representation of its (constant) inertia tensor expressed in $\PsiB$}, $m$ is the mass of the rigid body, and $g \in \mathbb{R}^3$ is the inverse direction of the gravitational acceleration vector in \ramy{$\PsiV$}.
The rate of change of the total energy is given by:
\begin{equation}\label{eq:Hdot_b}
    \dot{H}_\text{b}=\langle \partial_h H_\text{b} | \dot{h} \rangle_{T_h SE(3)} +\langle \partial_p H_\text{b} | \dot{p} \rangle_{\mathfrak{se}^*(3)},
\end{equation}
where the dual pairing notation $\langle | \rangle$ is embedded with a subscript indicating the vector space on which the pairing is implemented.
The rate of change of the state variables $\dot{x}_{\mathrm{b}}=(\dot{h},\dot{p})\in T_hSE(3) \times \mathfrak{se}^*(3)$ and the partial derivatives $\partial_{x_{\mathrm{b}}}H_\text{b}:=(\partial_h H_\text{b}, \partial_p H_\text{b})\in T_h^*SE(3) \times \mathfrak{se}(3)$ represent the flow and effort variables of the energy storage subsystem, respectively.
It can be shown that \ramy{\cite{rashad2021energy,Stramigioli2001,van1997interconnected}}
\begin{align}
    \partial_h H_\text{b}&=(0,mg)\in T_R^*SO(3) \times T_{\xi}\mathbb{R} \cong T_h^*SE(3) \\
    \partial_p H_\text{b} &=\cl{I}^{-1} p=:T^\text{b} \in \mathfrak{se}(3), \label{eq:Tb_def}
\end{align}
where $T^\text{b}\in \mathfrak{se}(3)$ denotes the twist of $\PsiB$ with respect to $\PsiV$ expressed in $\PsiB$.
This concludes the details of the energy storage $\mathbb{C}$-element of the pH model.
The equations of motion are encoded in the finite-dimensional Dirac structure $\cl{D}_\text{b}$ and are given by
\begin{align}
    \begin{pmatrix} 
    \dot{h} \\ \dot{p}
    \end{pmatrix}&=\begin{pmatrix}
    0 & \chi_h \\ -\chi^*_h & \cl{J}(p)
    \end{pmatrix}\begin{pmatrix}
    \partial_{h}H_\text{b} \\ \partial_{p}H_\text{b} \end{pmatrix}+\begin{pmatrix}
    0 \\ I_{6}
    \end{pmatrix}W^\text{b}, \label{eq:pH_rigid_1}\\
    T^\text{b}&=\begin{pmatrix}
    0 & I_6 \end{pmatrix}\begin{pmatrix}
    \partial_{h}H_\text{b} \\ \partial_{p}H_\text{b} \end{pmatrix}, \label{eq:pH_rigid_2}
\end{align}
where $W^\text{b}\in \mathfrak{se}^*(3)$ represents the external wrench (i.e., generalised force) applied \ramy{to the rigid body and expressed in $\PsiB$}, $I_6$ is the $6$-dimensional identity matrix, $\cl{J}(p)\in \mathbb{R}^{6 \times 6}$ is the skew-symmetric matrix:
\begin{equation*}
    \cl{J}(p)=\begin{pmatrix}
    \tilde{p}_{\omega} & \tilde{p}_v \\ \tilde{p}_v & 0
    \end{pmatrix}, \,\,\,\,\,\, p=\begin{pmatrix}
    p_\omega \\ p_v
    \end{pmatrix} \in \mathfrak{se}^*(3)\cong \mathbb{R}^6,
\end{equation*}
where \ramy{$p_\omega,p_v \in \mathbb{R}^3$ represent the angular and linear momenta of the rigid body expressed in $\PsiB$} and $\tilde{p}_{\omega}$, $\tilde{p}_v$ represent their skew-symmetric matrix counterparts given by (\ref{eq:skewsymIsom}).
Furthermore, the map $\chi_h: \mathfrak{se}(3) \to T_h SE(3)$ relates the Lie algebra element $T^\text{b} \in \mathfrak{se}(3)$ with the tangent vector $\dot{h}:=\chi_h(T^\text{b})\in T_h SE(3)$ by:
\begin{equation*}
    \dot{h}=(\dot{R},\dot{\xi})=(R\omega^\text{b},Rv^\text{b})\in T_RSO(3) \times \mathbb{R}^n \cong T_h SE(3),
\end{equation*}
\ramy{where $\omega^\text{b},v^\text{b} \in \bb{R}^3$ denote the linear and angular velocity parts of $T^\text{b}$.}
On the other hand, the dual map $\ramy{\chi}^*_h:T^*_hSE(3) \to \mathfrak{se}^*(3)$ relates the effort variable $\partial_h H_\text{b} \in T_h^*SE(3)$ to a wrench in $\mathfrak{se}^*(3)$ \ramy{and is defined implicitly by
\begin{equation}\label{eq:def_chi}
\langle  \chi^*_h(\Gamma_h) | T \rangle_{\mathfrak{se}(3)} = \langle \Gamma_h | \chi_h(T) \rangle_{T_h SE(3)} , \qquad \forall T\in \mathfrak{se}(3), \Gamma_h\in T_h^*SE(3).
\end{equation}}
Note that the term $-\chi_h^*(\partial_h H_\text{b})$ \ramy{in the momentum balance (\ref{eq:pH_rigid_1})} corresponds to the wrench due to gravity which is simply "minus" the gradient of a potential function.

\ramy{
The power balance encoded by the Dirac structure $\cl{D}_\text{b}$ defining the \pH model in (\ref{eq:pH_rigid_1}-\ref{eq:pH_rigid_2}) is given by the following result.
\begin{theorem}
Along solutions of the \pH system (\ref{eq:pH_rigid_1}-\ref{eq:pH_rigid_2}), the Hamiltonian function (\ref{eq:Ham_rigid_body}) satisfies the power balance:
$$\dot{H}_\text{b}=\langle W^\text{b} | T^\text{b} \rangle_{\mathfrak{se}(3)}.$$
\end{theorem}
\begin{proof}
By substituting (\ref{eq:pH_rigid_1}-\ref{eq:pH_rigid_2}) in (\ref{eq:Hdot_b}) and using the skew-symmetry of $\cl{J}$ and the definitions of $\chi_h^*$ in (\ref{eq:def_chi}) and $T^\text{b}$ in (\ref{eq:Tb_def}), we have that
\begin{align*}
\dot{H}_\text{b}= & \langle \partial_h H_\text{b} | \dot{h} \rangle_{T_h SE(3)} +\langle \partial_p H_\text{b} | \dot{p} \rangle_{\mathfrak{se}^*(3)}\\
=& \langle \partial_h H_\text{b} | \chi_h(\partial_{p}H_\text{b}) \rangle_{T_h SE(3)} + \langle \partial_p H_\text{b} | -  \chi^*_h(\partial_{h}H_\text{b}) + \cl{J}(p)\partial_p H_\text{b} + W^\text{b} \rangle_{\mathfrak{se}^*(3)}\\
=& \langle \partial_h H_\text{b} | \chi_h(\partial_{p}H_\text{b}) \rangle_{T_h SE(3)} - \langle  \chi^*_h(\partial_{h}H_\text{b}) | \partial_p H_\text{b} \rangle_{\mathfrak{se}(3)} + \langle  W^\text{b} |\partial_p H_\text{b}  \rangle_{\mathfrak{se}(3)}\\
=& \langle  W^\text{b} |T^\text{b} \rangle_{\mathfrak{se}(3)}.
\end{align*}
\end{proof}}
Note that the duality pairing between $W^\text{b} \in \mathfrak{se}^*(3)$ and $T^\text{b}\in \mathfrak{se}(3)$ above corresponds to the external power supplied to the rigid body from the entity generating the wrench $W^\text{b}$. The pair $(W^\text{b},T^\text{b})$ \ramy{define} the power port by means of which the rigid body can be interconnected to other systems to compose a bigger dynamical system.
For more details on this pH representation of rigid body motion and its derivation using Lie-Poisson reduction, the reader is referred to \cite{Stramigioli2001,rashad2021energy}.

Since the goal of this section is to show how to use the interaction port \ramy{$(W^\text{b},T^\text{b})$} to construct the FSI \ramy{\pH}model, it is important to introduce the mathematical tools making the open ports for the fluid system and the rigid body compatible for interconnection. 
\ramy{To this purpose, we introduce two additions to the rigid body \pH model presented above which are different representations of the interaction port $(W^\text{b},T^\text{b})$.
The first addition is a change of coordinates from $\PsiB$ to $\PsiV$ described by
\begin{equation}\label{eq:change_coord_twist_wrench}
T^\text{v} = Ad_h T^\text{b}, \qquad \qquad W^\text{b} = Ad_h^\top W^\text{v},
\end{equation}
where $\map{Ad_h}{\mathfrak{se}(3)}{\mathfrak{se}(3)}$ is the adjoint action of $SE(3)$ on $\mathfrak{se}(3)$ that takes the matrix representation
$$Ad_h = \TwoTwoMat{R}{0}{\tilde{\xi}R}{R}, \qquad \forall h = (R,\xi)\in SE(3).$$
By the definition of a matrix transpose, it is straightforward to show that
\begin{equation}\label{eq:MTF_PsiV}
\langle  W^\text{b} |T^\text{b} \rangle_{\mathfrak{se}(3)} = \langle  W^\text{v} |T^\text{v} \rangle_{\mathfrak{se}(3)}.
\end{equation}
}

\begin{figure}
\centering
\includegraphics[width = 0.7\textwidth]{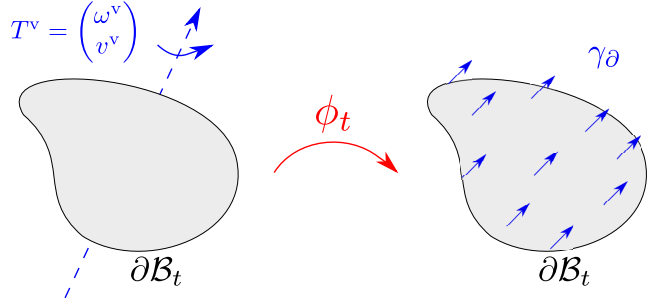}
\caption{Illustration of the action of the reconstruction map $\phi_t$ associating to a twist $T^\text{v}$ a vector-field $\gamma_\partial$ on the boundary $\partial\cl{B}_t$.}
\label{fig_reconstruction_map}
\end{figure}

\ramy{For the second addition,} we introduce the "reconstruction map"
\begin{equation*}
    \phi_t:\mathfrak{se}(3) \to \spVecX{\cl{B}_t}{} \tens_{i_b} \spKForm{0}{\partial \cl{B}_t}  
\end{equation*}
which, at any time instant $t$, reconstructs from the finite-dimensional twist $T^\text{v} \in \mathfrak{se}(3)$ its corresponding infinite-dimensional vector field $\phi_t(T^\text{v})$ defined at the surface of the rigid body, as depicted in Figure \ref{fig_reconstruction_map}.
The components of the vector field $\phi_t(T^\text{v})=: \gamma_{\partial}$ in \ramy{the inertial frame $\PsiV$} are given by
\begin{equation*}
    \gamma_{\partial}^i=-\tilde{q}^i_k \omega^k +v^i,\qquad i,k\in\{1,2,3\},
\end{equation*}
where Einstein sum convention on repeated indices is used.
Equivalently, using vector notation
\begin{equation*}
    \pmb{\gamma_\partial}=-\tilde{\pmb{q}}^\text{v} \omega^\text{v} + v^\text{v}=\tilde{\omega}^\text{v} \pmb{q}^\text{v}+v^\text{v},
\end{equation*}
where $\pmb{q}^\text{v}\in \mathbb{R}^3$ denotes the coordinates of the point $q \in \partial \cl{B}_t$ in $\PsiV$ and $\tilde{q}^i_k \in \mathbb{R}$ are the components of the skew-symmetric matrix $\tilde{\pmb{q}}^\text{v} \in \mathfrak{so}(3)$. Note hat for $n=3$ we have $\tilde{\pmb{q}}\omega=\pmb{q} \times \omega$ which represents the vector product in $\mathbb{R}^3$.

\ramy{\begin{remark}
\label{rem:lierem}
In other works in the literature, e.g. \cite{glass2012movement,planas2014viscous}, the vector field $\pmb{\gamma_\partial}$ is introduced as $\pmb{\gamma_\partial} = \tilde{\omega}^\text{v}(\pmb{q}^\text{v} - \xi) + \dot\xi$ which is equivalent to the form presented above using the identity $v^\text{v} = \dot{\xi} - \dot{R}R^\top \xi = \dot{\xi} -\tilde{\omega}^\text{v}\xi$.
\end{remark}}

In summary, the reconstruction map $\Psi_t$ allows \stefano{to} represent the twist $T^\text{v} \in \mathfrak{se}(3)$ as the vector-valued zero-form $\phi_t(T)$ on the $(n-1)$-dimensional manifold $\partial \cl{B}_t$ of the rigid body. \stefano{Such a concept could be also defined without the use of coordinates and using the} \ramy{induced action of} \stefano{$se(3)$ on the Eucledian space} \ramy{$\bb{R}^3$}.
By duality, the dual map 
\begin{equation*}
    \phi_t^*:\spKForm{1}{\cl{B}_t} \tens_{i_b} \spKForm{n-1}{\partial \cl{B}_t} \to \mathfrak{se}^*(3)
\end{equation*}
allows representing any covector-valued $(n-1)$-form $\alpha_{\partial}$ as a finite-dimensional wrench $\phi_t^*(\alpha_{\partial})\in \mathfrak{se}^*(3)$ defined implicitly \ramy{for any $T^\text{v} \in \mathfrak{se}(3)$} by
\begin{equation}
\label{eq:dualityPsi}
    \langle \phi_t^*(\alpha_{\partial}) | T^\text{v} \rangle_{\mathfrak{se}(3)}=\int_{\partial \cl{B}_t} \alpha_{\partial} \dot{\wedge} \phi_t(T^\text{v}).
\end{equation}
By letting $W^\text{v}:=\phi_t^*(\alpha_{\partial})$, (\ref{eq:dualityPsi}) can be rewritten as
\begin{equation}\label{eq:MTF_reconstr}
    \langle W^\text{v}|T^\text{v} \rangle_{\mathfrak{se}(3)}=\int_{\partial \cl{B}_t} \alpha_{\partial} \dot{\wedge} \gamma_{\partial},
\end{equation}
which states that the power flowing through the finite-dimensional port $(W^\text{v},T^\text{v})$ is equal to that flowing through the infinite-dimensional port $(\alpha_\partial, \gamma_\partial)$. 

\ramy{The two power balances (\ref{eq:MTF_PsiV}) and (\ref{eq:MTF_reconstr}) and their corresponding maps are characterised graphically in Fig. \ref{fig_pH_body_bond_graph} by the two modulated transformers with the symbol $\mathbb{MTF}$}, similarly to what was introduced in \cite{Mahony} for a visual application.
\ramy{The first transformer is modulated in the sense that it requires, at a given time instant, the current configuration $h\in SE(3)$ of the rigid body to change coordinates of the external wrench and the body twist by (\ref{eq:change_coord_twist_wrench}).
Similarly, the second} transformer is modulated because it requires explicitly all points $q\in \partial \cl{B}_t$, at a given time instant, to calculate the vector field $\gamma_\partial$ given a twist $T$ and to calculate the wrench $W$ given the covector-valued form $\alpha_\partial$.

\ramy{The current configuration $h_t:=h(t)$ of the rigid body is determined by the solution of (\ref{eq:pH_rigid_1}).}
Let $\cl{B}_0$ denote the reference configuration of the rigid body at $t=0$ and $\partial \cl{B}_0$ denote its boundary. \ramy{At every time instant, $h_t=(R_t,\xi_t) \in SE(3)$, being a positive isometry on $\cl{V}$, maps any $q\in \partial \cl{B}_0$ to its location $h_t(q):=R_t(q)+\xi_t$ at time $t$.} The same applies to the rigid body's interior points in $\cl{B}_0$. Thus, at any time $t$, $\cl{B}_t:=h_t(\cl{B}_0)$ and $\partial \cl{B}_t:=h_t(\partial \cl{B}_0)$ are defined as the images of $\cl{B}_0$ and $\partial \cl{B}_0$ under $h_t$, respectively. We denote this action of $SE(3)$ by \ramy{$\Phi$, as shown in Fig. \ref{fig_pH_body_bond_graph}.
This concludes the \pH model of a floating rigid body in a gravitational field.}

\subsection{Port-Hamiltonian Fluid-Structure Interaction}
Finally, using all the presented constructions, we will now interconnect both the pH model of the incompressible viscous flow with that of the rigid body to compose a complete dynamical model describing fluid-structure interaction, depicted in Fig. \ref{fig_pH_fsi_bond_graph}.

\begin{figure}
\centering
\includegraphics[width = \textwidth]{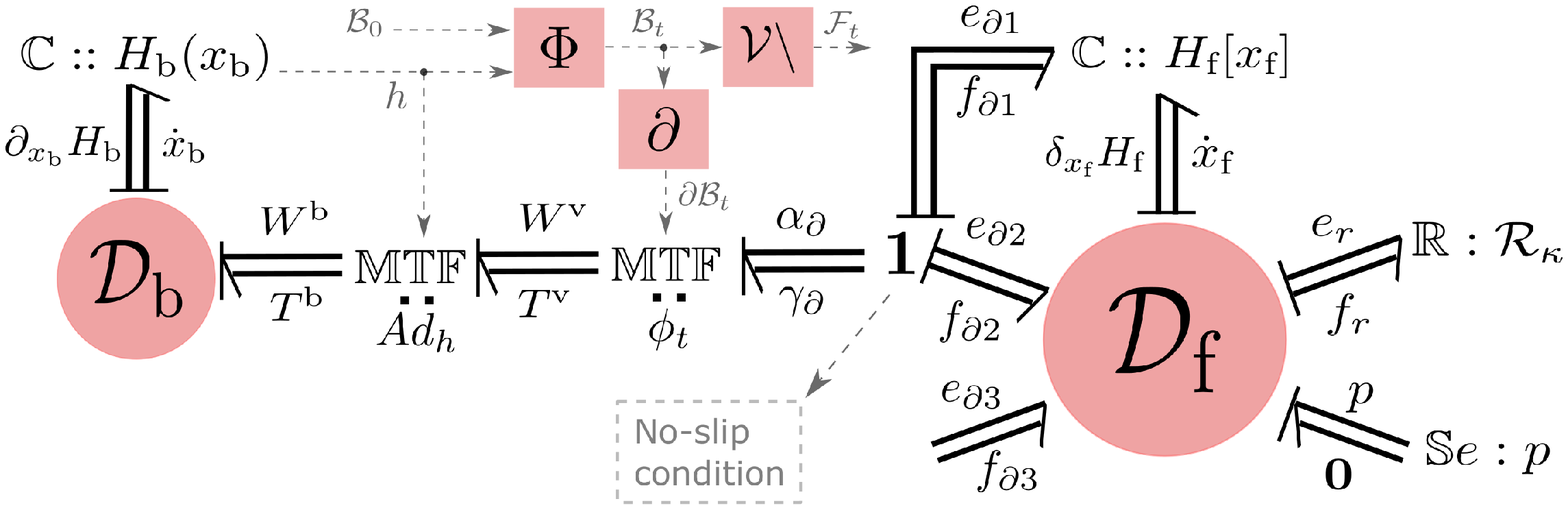}
\caption{Bond-graph representation of the combined port-Hamiltonian models describing the fluid-structure interaction.}
\label{fig_pH_fsi_bond_graph}
\end{figure}

By comparing Fig. \ref{fig_pH_fsi_bond_graph} to Fig. \ref{fig_pH_fluid_bond_graph} and \ref{fig_pH_body_bond_graph}, it is straightforward to see the compositional nature of the pH framework, which allows to combine the models of the distributed-parameter fluid dynamical system with the lumped-parameter rigid body dynamics by only specifying how the two systems exchange energy though the boundary $\partial \cl{B}_t$. This energy exchange is characterised by the common flow constraint, graphically represented by the $1$-junction in Fig. \ramy{\ref{fig_pH_fsi_bond_graph}}. This constraint acts on the three ports $(e_{\partial 1}, f_{\partial 1}),(e_{\partial 2},f_{\partial 3})$ and $(\alpha_\partial,\gamma_{\partial})$ by means of:
\begin{equation*}
    f_{\partial 1}=f_{\partial 2}=\gamma_{\partial}, \qquad \qquad \alpha_\partial + e_{\partial 1} +e_{\partial 2}=0.
\end{equation*}
In words, the flow variables are equivalent while the effort variables sum to zero. Using the definition of $f_{\partial i}$ and $e_{\partial i}$ in Table \ref{tab_pH_port_variables}, we can rewrite the above constraints as
\begin{align}
    &\textrm{ptr}_{b}(u)=\textrm{ptr}_{b}(v)=\gamma_\partial, \\
    &\alpha_\partial=\text{ptr}_{b}(\cl{T})=\text{ptr}_{b}(- \star p +\cl{T}_{\kappa}).
\end{align}
The flow constraint describe\ramy{s} the no-slip condition, stating that the velocity of the fluid at $\partial \cl{B}_t$ ($\textrm{ptr}_{b}(v)$) is equal to the velocity of the boundary itself ($\textrm{ptr}_{b}(u)$), which is generated by the rigid body motion ($\gamma_\partial$). On the other hand, the effort constraint states that the surface stress on the rigid body's boundary ($\alpha_{\partial}$) is \ramy{equal to} the total stress tensor ($\text{ptr}_{b}(\cl{T})$) \ramy{which includes the} static pressure and shear stress. Notice that the dynamic pressure ($\text{ptr}_{b}(\cl{H}_{\text{f}})$) component cancels out due to the same reason previously seen in (\ref{eq:relativeFlux}), i.e. the flow constraint makes the fluid follow the body (Lagrangian description) at $\partial \cl{B}_t$, eliminating energy transfer due to advection.
\ramy{On the other hand, this term} remains rightfully present at the boundary $\partial \cl{V}$, where the representation of the fluid keeps being of Eulerian type. 

Now we prove that the above constraints on the two pH models correctly describe the complete fluid-body system by computing the wrench applied on the rigid body's surface caused by the fluid. This wrench is given by the following result.

\begin{theorem}
Consider the wrench $\ramy{W^\text{\normalfont{v}}} \in \mathfrak{se}^*(3)$ given by 
\begin{equation*}
    \ramy{W^\text{\normalfont{v}}}=\ramy{\phi}^*_t(\alpha_\partial)=\ramy{\phi}^*_t(\text{ptr}_{b}(- \star p +\cl{T}_{\kappa})).
\end{equation*}
Let $\ramy{\tau,f} \in \mathbb{R}^3$ denote the torque and force parts, respectively, of the wrench. Furthermore, let $\pmb{n}$ (with components $n^i$) denote the normal vector field to $\partial \cl{B}_t$ and $\volF^{\partial 
\cl{B}}$ the induced volume form on the $(n-1)$-dimensional manifold $\partial \cl{B}_t$. The components of $\tau$ and $f$ (indicated with indices down since they are co-vectors) are then given by
\begin{align*}
    f_i&=\int_{\partial \cl{B}_t}(-p \cdot \delta_{ij}+\sigma_{ij})n^i \volF^{\partial \cl{B}}, \\
    \tau_i &=\int_{\partial \cl{B}_t} (-p \cdot \delta_{ij}+\sigma_{ij})\tilde{q}^i_k n^k \volF^{\partial \cl{B}},
\end{align*}
where $\delta_{ij}$ are the kronecker delta symbols and $\sigma_{ij}$ are the Cauchy stress tensor matrix components in Eucledian space, related by (\ref{eq:DefShearViscosity}) by $\cl{T}_{\kappa}=\star_2 \sigma$, considering $\sigma$ as covector-valued one-form.
\begin{proof}
In order to compute the result, we need to manipulate the implicit expression (\ref{eq:dualityPsi}) in order to extract a closed form expression for $\ramy{\phi}^*_t$:

\begin{align*}
    \underbrace{\tau_j \omega^j+f_j v^j}_{\langle W|T \rangle_{\mathfrak{se}(3)}}&=\int_{\partial \cl{B}_t} \alpha_\partial \dot{\wedge} \gamma_\partial \\
    &=\int_{\partial \cl{B}_t}(-\star p + \star_2 \sigma)|_{\partial \cl{B}_t} \dot{\wedge} \gamma_\partial \\
    &=\int_{\partial \cl{B}_t} (-p \cdot \delta_{ij}+\sigma_{ij})n^i\gamma^j \volF^{\partial \cl{B}} \\
    &=\int_{\partial \cl{B}_t} (-p \cdot \delta_{ij}+\sigma_{ij})n^i(-\tilde{q}^j_k \omega^k+v^j) \volF^{\partial \cl{B}} \\
    &=\int_{\partial \cl{B}_t} [(-p \cdot \delta_{ij}+\sigma_{ij})n^i v^j +(-p \cdot \delta_{ij}+\sigma_{ij}) \tilde{q}^i_k n^k \omega^j]\volF^{\partial \cl{B}}
\end{align*}
which \ramy{proves} the result.
\end{proof}
\end{theorem}

For a sake of completeness \ramy{and comparison with other works e.g. \cite{planas2014viscous}} we report the vector calculus notation of the computed wrench:
\begin{align*}
    \Vec{f}&=\int_{\partial \cl{B}_t}(-p I_n+\Sigma)\Vec{n} \volF^{\partial \cl{B}}, \\
    \Vec{\tau} &=\int_{\partial \cl{B}_t} (-p I_n + \Sigma)\pmb{\tilde{q}} \Vec{n} \volF^{\partial \cl{B}},
\end{align*}
where $\Sigma$ is the matrix representation of the 2-rank Cauchy stress tensor having as entries $\sigma_{ij}$. We highlight that the expression for the force coincides with the works following vector calculus-based derivations like \cite{planas2014viscous}, while the expression for the torque varies since we are not deriving the components of the torque in an inertial frame with respect to an origin, but part of the geometric wrench. The two expressions can be related with an argument similar to that discussed in Remark \ref{rem:lierem}.

\section{Conclusions and future work}
\label{sec:conc}
In this paper we extended the geometric port-Hamiltonian formulation for infinite-dimensional systems to the case in which the spatial domain of the underlying PDE is moving in time. We introduce a novel duality to define the power port corresponding to the moving domain mechanism, which uses the technology of vector-valued forms, in contrast to the standard formulation which uses scalar-valued forms. We demonstrate how the novel defined duality is necessary in order to represent the no-slip condition, which made it possible to give a covariant representation of a fluid-structure interaction system in the port-Hamiltonian framework.

As future work we are working on extending the theoretical model to the case of FSI involving elasticity in the solid, since the presented procedure will allow to produce the final model by only changing the dynamic model of the solid, while the fluid model and the interconnection procedure would apply the same way. Furthermore we are researching numerical techniques able to integrate the presented FSI system by conveniently exploiting the geometric port-Hamiltonian structure.

\section*{Funding}
This work was supported by the PortWings project funded by the European Research Council [Grant Agreement No. 787675]

\bibliography{refs}

\end{document}